\documentclass[12pt]{amsart}
\usepackage{url}
\usepackage{graphicx}
\usepackage{amssymb, amscd}
\usepackage{xy}
\usepackage{bm}
\xyoption{all}

\usepackage[OT2,OT1]{fontenc}
\newcommand\cyr{%
 \renewcommand\rmdefault{wncyr}%
 \renewcommand\sfdefault{wncyss}%
 \renewcommand\encodingdefault{OT2}%
 \normalfont
\selectfont}
 \DeclareTextFontCommand{\textcyr}{\cyr}

 \def\cprime{\char126}

\usepackage{pinlabel}
\usepackage[linktocpage]{hyperref}

\newcommand{\qhe}{Q_{h,e}}
\newcommand{\qhs}{$\Q$H$S^3$}

\DeclareMathOperator{\rank}{rank}

\DeclareMathOperator{\degree}{deg}

\newcommand\altau{\alpha^\tau}
\def\co{\colon\thinspace}

\newcommand{\xtilde}{\widetilde{X}}

\newcommand{\ctilde}{\widetilde{C}}
\newcommand{\qtilde}{\widetilde{Q}}

\newcommand{\ztilde}{\widetilde{Z}}

\newcommand{\lk}{\operatorname{lk}}
\newcommand{\PD}{\operatorname{PD}}

\newcommand{\sign}{\operatorname{sign}}

\newcommand{\C}{\mathbb C}

\newcommand{\Z}{\mathbb Z}
\newcommand{\Q}{\mathbb Q}
\newcommand{\R}{\mathbb R}

\newcommand{\RP}{\mathbb{RP}}

\newcommand{\CPbar}{\overline{\mathbb{CP}}}

\newcommand{\pont}{\widetilde{{\mathcal P}}}

\newcommand{\OO}{\mathcal O}

\newcommand{\la}{\langle}
\newcommand{\ra}{\rangle}

\def\spinc{\ifmmode{\textrm{Spin}^c}\else{$\textrm{Spin}^c$}\fi}
\newcommand{\spincs}{\mathfrak s}
\newcommand{\spinct}{\mathfrak t}

\newcommand{\spincu}{\mathfrak u}
\newcommand{\spincv}{\mathfrak v}
\newcommand{\spincw}{\mathfrak w}

\newcommand{\tspincs}{\tilde{\mathfrak s}}
\newcommand{\spincx}{\mathfrak x}

\def\oz{Ozsv{\'a}th--Szab{\'o}}
\renewcommand{\phi}{\varphi}

\newtheorem{theorem}{Theorem}[section]
\newtheorem{thm}{Theorem}

\newtheorem{cor}[thm]{Corollary}

\newtheorem{lemma}[theorem]{Lemma}

\newtheorem{proposition}[theorem]{Proposition}
\newtheorem{corollary}[theorem]{Corollary}
\theoremstyle{definition}
\newtheorem{definition}[theorem]{Definition}
\newtheorem{remark}[theorem]{Remark}

\newtheorem{example}[theorem]{Example}

\makeatletter
\let\@wraptoccontribs\wraptoccontribs
\makeatother

\usepackage[dvipsnames]{xcolor}

\DeclareMathOperator{\Spin}{Spin}
\def\s{\mathfrak{s}}
\def\t{\mathfrak{t}}
\def\CF {\operatorname{CF}}
\def\CFhat {\operatorname{\widehat{CF}}}
\def\HF {\operatorname{HF}}
\def\HFhat {\operatorname{\widehat{HF}}}
\def\HFp {\operatorname{HF}^+}

\def\HFi {\operatorname{HF}^\infty}

\def\HFKhat {\operatorname{\widehat{HFK}}}

\def\dtop {d_{\operatorname{top}}}
\def\dbot {d_{\operatorname{bot}}}
\DeclareMathOperator{\Tors}{Tors}
\DeclareMathOperator{\gr}{gr}
\DeclareMathOperator{\im}{im}
\DeclareMathOperator{\grbar}{\widetilde{\gr}}
\renewcommand{\top}{{\operatorname{top}}}
\renewcommand{\bot}{{\operatorname{bot}}}

\newcommand{\abs}[1] {\left\lvert #1 \right\rvert}

\newcommand{\gen}[1] {\langle #1 \rangle}
\def\conn{\mathbin{\#}}
\def\bconn{\mathbin{\natural}}
\def\minus{\smallsetminus}

\newcommand{\veps}{\varepsilon}

\newcommand{\arxiv}[1]{\href{http://arxiv.org/abs/#1}{arXiv:#1}}

\title{Non-orientable surfaces in homology cobordisms}

\author[Adam Simon Levine]{Adam Simon Levine}
\address{Department of Mathematics, \newline\indent Princeton University \newline\indent Princeton, NJ 08540}
\email{asl2@math.princeton.edu}
\author[Daniel Ruberman]{Daniel Ruberman}
\address{Department of Mathematics, MS 050\newline\indent Brandeis
University \newline\indent Waltham, MA 02454}
\email{ruberman@brandeis.edu}
\author[Sa\v so Strle]{Sa\v so Strle}
\address{Faculty of Mathematics and Physics \newline\indent
University of Ljubljana \newline\indent Jadranska 21 \newline\indent
1000 Ljubljana, Slovenia }
\email{saso.strle@fmf.uni-lj.si}
\contrib[with an appendix by]{Ira M.~Gessel}
\address{Department of Mathematics, MS 050\newline\indent Brandeis
University \newline\indent Waltham, MA 02454}
\email{gessel@brandeis.edu}
\thanks{The first author was supported by an NSF Postdoctoral Fellowship.  The second author was partially supported by NSF Grant 1105234.   Visits of the authors were supported by a Slovenian-U.S.A. Research Project BI-US/09-12-004, and by the Simons Center.\\
Math.~Subj.~Class.~2010: 57M25 (primary), 57Q60 (secondary).}

\begin{document}
\begin{abstract} We investigate constraints on embeddings of a non-orientable surface in a $4$-manifold with the homology of $M \times I$, where $M$ is a rational homology $3$-sphere.  The constraints take the form of inequalities involving the genus and normal Euler class of the surface, and either the Ozsv\'ath--Sazb\'o $d$-invariants~\cite{oz:boundary} or Atiyah--Singer $\rho$-invariants~\cite{aps:II} of $M$.  One consequence is that the minimal genus of a smoothly embedded surface in $L(2p,q) \times I$ is the same as the minimal genus of a surface in $L(2p,q)$.  We also consider embeddings of non-orientable surfaces in closed $4$-manifolds.
\end{abstract}
\maketitle

\section{Introduction}
Although a non-orientable surface cannot embed in the $3$-sphere, an orientable $3$-manifold $M$ can contain non-orientable surfaces, as long as $H_1(M;\Z_2)$ is non-zero.  A classic paper of Bredon and Wood~\cite{bredon-wood:surfaces} determines the minimal genus of such a surface in a lens space $L(2p,q)$, where the genus $h$ of a connected, non-orientable surface $F$ is defined as $h = b_1(F) = 2-\chi(F)$.   We extend this investigation, using both classical techniques and Heegaard Floer gauge theory, to the setting where the surface is embedded in the interior of a homology cobordism $W$ between rational homology $3$-spheres (\qhs) $M_0$ and $M_1$.  This of course includes the special case of a product $W=M^3 \times I$ for $M$ a \qhs, and the special case when $M$ is a lens space is of particular interest in light of the Bredon--Wood results.  The technique we employ to bound the genus via the G-signature theorem stems from Massey's paper~\cite{massey:whitney}; our use of gauge theory is related to the paper of T.~Lawson~\cite{lawson:rp2} and can be used to reproduce and extend the results of that paper.   A similar combination of techniques appears in the recent preprint of Batson~\cite{batson:non-orientable}, which investigates the non-orientable $4$-ball genus of a knot in the $3$--sphere.

Two differences between the $4$-dimensional setting and the $3$-dimensional one are worth noting.  The first is that since $\RP^2$ embeds in $4$-space, there is no lower bound for the genus of an embedding in an arbitrary $4$--manifold without some additional hypothesis.   Unless explicitly stated to the contrary, we henceforth assume that our embeddings are \emph{essential} in the sense that
\begin{equation}\label{nontriv}\tag{*}
\textrm{The inclusion map}\ j_*\co H_2(F;\Z_2) \to H_2(W;\Z_2)\ \textrm{is non-trivial.}
\end{equation}
The second is that the normal bundle of an embedding in a $4$-manifold is not determined by the homology class that it carries. The normal bundle is determined by the normal Euler class, which is an integer that we denote by $e$; see section~\ref{S:prelim} below for details.  For instance, there are two standard embeddings of $\RP^2$ in $S^4$, with Euler classes $\pm 2$.  Except in sections~\ref{S:prelim}, \ref{S:sign}, and \ref{S:top}, which apply equally to locally flat embeddings in topological manifolds, all manifolds and embeddings will be assumed to be smooth.

Our main result gives a bound for the genus of an essential surface in a homology cobordism $W$ between rational homology spheres $M_0$ and $M_1$. It is stated in terms of the $d$-invariants defined by Ozsv{\'a}th and Szab{\'o}~\cite{oz:boundary}; these are a collection of rational numbers associated to spin$^c$ structures on a rational homology sphere.  Equation~\eqref{nontriv} implies that the homology class $j_*([c]) \in H_1(W) = H_1(M_0)$ is non-trivial, where $[c]$ is the unique torsion class in $H_1(F)$. The Poincar\'e dual of this class is a $2$-torsion class $\phi \in H^2(M_0)$.  Recalling that the $2$-dimensional cohomology acts on the set of spin$^c$ structures on $M_0$, we define
\[
\Delta=\Delta(M_0,\phi)=\max\{d(M_0,\spincs+\phi)-d(M_0,\spincs)\mid \spincs\in \spinc (M_0)\}
\]
which is an element of $\frac12\Z$.

We prove the following, our main result, in Section~\ref{S:Hcob}.
\begin{thm} \label{T:Mbound}
Suppose that $W$ is a homology cobordism between rational homology spheres $M_0$ and $M_1$, and that $F_h\subset W$ is essential and has normal Euler number $e$. Let $\Delta=\Delta(M_0,\phi)$. Then
\[
h \ge 2\Delta, \quad |e| \le 2h - 4\Delta, \quad \text{and} \quad e \equiv 2h - 4\Delta \pmod4.
\]
\end{thm}

For a lens space $L(2k,q)$, there is only one choice for the class $\phi$, and the invariant $\Delta$ turns out to be equal to one-half of the minimal genus function $N(2k,q)$ determined by Bredon and Wood. (This follows from the work of Ni and Wu on rational genus \cite{ni-wu:rational-genus}, as described in section~\ref{S:rational}; a purely number-theoretic proof appears in the appendix to this paper, written by Ira Gessel.) Thus Theorem~\ref{T:Mbound} implies that the minimum genus problem in $L(2k,q) \times I$ is the same as that in $L(2k,q)$. More precisely, we have:
\begin{cor} \label{C:lens}
Let $W$ be any homology cobordism from $L(2k,q)$ to itself (e.g. $W = L(2k,q) \times I$). Let $N = N(2k,q)$. There is an essential embedding of $F_h$ in $W$ with normal Euler number $e$ if and only if
\[
h \ge N, \quad \abs{e} \le 2(h - N) \quad \text{and} \quad e \equiv 2(h - N) \pmod4.
\]
In other words, $F_h$ has the same genus and normal Euler number as the connected sum of an embedded, non-orientable surface in $L(2k,q) \times \{\frac12\}$ with zero or more copies of $\RP^2 \subset S^4$.
\end{cor}
See Corollary \ref{C:simple} below for a more general statement.

The idea of twisting a spin$^c$ structure also works in a closed definite $4$-manifold, and we obtain bounds for the genus of a smoothly embedded surface $F$ in terms of the Euler class and mod $2$ homology class carried by $F$.  In the special case that the surface is Poincar\'e dual to $w_2$, we get such bounds without the assumption that the manifold be definite, using Furuta's 10/8 theorem~\cite{furuta:118}.

The paper is organized as follows.  In Section~\ref{S:prelim}, we explain the basic topological mechanism (`twisting') behind our obstructions and establish a useful congruence for the genus.  We give a brief exposition of the Heegaard Floer correction terms associated to torsion spin$^c$ structures on a $3$-manifold with $b_1 >0$ in Section~\ref{S:correction}. In Section~\ref{S:bound}, we obtain obstructions to embeddings involving the values of these correction terms for circle bundles with orientable total space over non-orientable surfaces, which are then computed in Section~\ref{S:floercalc}. We assemble these ingredients in Section \ref{S:Hcob} to prove Theorem \ref{T:Mbound}.
In Section~\ref{S:rational}, we explain the connection of our work with that of Ni and Wu~\cite{ni-wu:rational-genus} on rational genus, leading to the proof of Corollary \ref{C:lens}. In Section~\ref{S:sign} we show how twisting combines with classical topological techniques stemming from the G-signature theorem to give further embedding obstructions in homology cobordisms, and then in Section~\ref{S:top} construct some locally flat embeddings that cannot be realized smoothly. We extend our results to the setting of closed $4$-manifolds in Section~\ref{S:definite}. The appendix by Ira Gessel provides a number-theoretic proof that the invariant $\Delta$ for lens spaces agrees with the Bredon--Wood minimal genus function.

\subsection*{Acknowledgements}
We would like to thank Josh Batson, Josh Greene, Jonathan Hanselman, Yi Ni, Peter Ozsv{\'a}th, and Nikolai Saveliev for helpful conversations and correspondence in the course of developing this work.

\section{Topological preliminaries}\label{S:prelim}
We will denote by $F=F_h$ the connected sum of $h$ copies of $\RP^2$; the integer $h$ is often referred to as the genus in this setting.
If $j\co F \to X$ is an embedding into an oriented manifold, then the normal bundle $\nu(F)$ satisfies the relation $\nu(F) \oplus TF \cong j^*TX$, and hence $w_1(\nu(F)) = w_1(F)$.  Applying the Whitney sum formula yields that $w_2(\nu) + w_1(F)^2 + w_2(F) = j^* w_2(X)$.  But it is well known that $w_1(F)^2 + w_2(F) = 0$ for any surface, and so $w_2(\nu)= j^* w_2(X)$.

Since the normal bundle is non-orientable, it has no Euler class in the usual sense.  On the other hand, a choice of orientation for a fiber of $\nu(F)$ at a point $x\in F$ determines an orientation for $T_x(F)$. This implies that one can define an integer-valued normal Euler number $e=e(F,X)$ of $F$ in $X$ by summing the local intersection numbers of $F$ with a nearby transverse copy of $F$.  The normal Euler class (and hence, number) may equally be defined as a twisted cohomology class, as detailed in~\cite[Appendix 1]{massey:whitney}.  The mod $2$ reduction of $e$ agrees with the second Stiefel-Whitney number of the normal bundle, because
$\langle w_2(\nu),[F]\rangle$ may be computed as a mod $2$ intersection number. Consequently, if $X$ is spin, as is the case for the homology cobordisms we consider, the Euler number of $\nu$ is even.

We adopt the following notational conventions.  $W$ will always denote an oriented homology cobordism with boundary components rational homology spheres, oriented so that $\partial W = M_1 \sqcup -M_0$.    Unless it is needed for clarity, we usually do not mention the inclusion $j$ and simply write $[F] \in H_2(W;\Z_2)$.
The homology and cohomology groups of $M_0$, $M_1$, and $W$ are isomorphic, and we generally use the same letter to indicate elements in these groups that correspond under the inclusion maps.  The same convention will apply to spin$^c$ structures and their associated $d$-invariants.
The notation $\PD$ will represent the Poincar\'e duality isomorphism, with a subscript indicating the manifold as necessary.

There is a unique $\R^2$ bundle over $F=F_h$ with $w_1 = w_1(F)$ and Euler number $e$; its total space is an oriented manifold.  The associated disk bundle will be denoted $P=P_{h,e}$, and its oriented boundary will be denoted $Q=Q_{h,e}$.
The complement of the interior of the normal bundle of $F$ in $W$ will be denoted $V$; keeping track of orientations we have that $\partial V = M_1 - M_0 - Q$.
We will need some basic topological properties of $Q$.

\begin{lemma}
\label{L:Q}
Let $Q=Q_{h,e}$ be the circle bundle over $F_h$ with $w_1 = w_1(F)$ and Euler number $e$. Then there is a short exact sequence
$$0\to \Z_2 [f] \to H_1(Q;\Z) \to H_1(F;\Z) \to 0,$$
where $[f]$ represents the class of the fiber circle. It follows that
$$H_1(Q;\Z)\cong \Z_2[f]\oplus H_1(F;\Z)$$
for $e$ even, and
$$H_1(Q;\Z)\cong \Z_4\oplus  H_1(F;\Z)/\Tors$$
for $e$ odd, where in the latter case $[f]$ represents twice the generator of the torsion subgroup.
\end{lemma}

\proof
We first compute the cohomology of $Q$ over $\Z_2$.  The Gysin sequence
for the projection $\pi\co Q \to F$ reads
\begin{equation}\label{gysin}
\xymatrix{
0 \ar[r] & H^1(F;\Z_2) \ar[r]^{\pi^*} & H^1(Q;\Z_2) \ar[r]^{S} & H^0(F;\Z_2) \ar[r]^{\kern8mm \smile w_2} & \\
& H^2(F;\Z_2) \ar[r]^{\pi^*} & H^2(Q;\Z_2) \ar[r]^{S} & H^1(F;\Z_2) \ar[r] & 0
}
\end{equation}
The map $S\co H^1(Q;\Z_2) \to H^0(F;\Z_2) \cong \Z_2$ is given by evaluation of a cohomology class on the circle fiber. If $e$ is even, then $w_2=0$ and the sequence splits into two short exact sequences showing that $H^1(Q;\Z_2)\cong \Z_2^{h+1}$. If $e$ is odd, then multiplication by $w_2$ is an isomorphism and $H^1(Q;\Z_2)\cong \Z_2^{h}$.

The long exact sequence (now with integer coefficients) of the pair $(P, Q)$ reduces to
$$0\to H_2(P,Q) \to H_1(Q) \to H_1(F) \to 0$$
where the image of $H_2(P,Q) \cong \Z_2$ is the $\Z_2$ class $f$ carried by the fiber.  Comparing, via the universal coefficient theorem, this result and the calculation of the $\Z_2$ cohomology implies that for $e$ even this sequence splits.  Similarly, for $e$ odd we get a $\Z_4$ extension of the two torsion groups with $\Z_2[f]$ representing a subgroup.
\endproof

An important part of the information about the embedding of $F$ in $W$ we use to obtain constraints on such embeddings is the existence of a `twisting' cohomology class on the complement of the surface, established in Proposition~\ref{twist} below.  To state this, recall that the torsion subgroup of $H_1(F)$ contains one non-trivial element (of order $2$) that we will denote by $c$; the image of $c$ in $H_1(W)$ (and the corresponding elements of $H_1(M_i)$) will be denoted by $[c]$.
The coefficient exact sequence $0 \to \Z \to \Z \to \Z_2 \to 0$ determines a Bockstein homomorphism $\beta \co H_2(\;\cdot\;;\Z_2) \to H_1(\;\cdot\;;\Z)$ and a corresponding Bockstein in cohomology, also denoted by $\beta$.
\begin{lemma}\label{L:restrict}
Let $F \subset W$ be essentially embedded in the homology cobordism $W$. Then
$\beta([F])=[c]$ in $H_1(W;\Z)$ and the restriction homomorphism
$$
H^2(W;\Z) \to H^2(P;\Z) \cong H^2(F;\Z) \cong \Z_2
$$
is given by reduction modulo $2$ and evaluation on $[F]\in H_2(W;\Z_2)$.
\end{lemma}
\begin{proof}
Note that $\beta \co H_2(F;\Z_2) \to H_1(F;\Z)$ sends $[F]$ to $c$.  Comparing the Bockstein sequences in homology for $F$ and $W$
$$
\xymatrix @C=0.8in @R=0.5in{
 0 \ar[d] \ar[r]&H_2(F;\Z_2) \ar[d] \ar[r]^{\beta}_{} & H_1(F;\Z) \ar[d]\\
 0  \ar[r] &H_2(W;\Z_2) \ar[r]^{\beta}_{} & H_1(W;\Z)
}
$$
(where we have used that $H_2(W;\Z) = H_2(M_0;\Z) = 0$) gives that $\beta([F])=[c]$ and that the torsion subgroup of $H_1(F;\Z)$ injects into $H_1(W;\Z)$.

Comparing the Bockstein sequences in cohomology gives
$$
\xymatrix @C=0.8in @R=0.5in{
H^2(W;\Z) \ar[d] \ar[r]^{\mod 2}_{} & H^2(W;\Z_2) \ar[d]\\
H^2(F;\Z) \ar[r]^{\mod 2}_{\cong} & H^2(F;\Z_2)
}
$$
which proves the statement about the restriction $H^2(W;\Z) \to H^2(P;\Z)$.
\end{proof}

\begin{proposition}\label{twist}
There is an element $\gamma \in H^2(V, M_0; \Z)$ of order two that restricts to
$\PD_{M_1}([c])$ in $H^2(M_1;\Z)$
and to $\PD(\tilde c)$ in $H^2(Q;\Z)$, where $\tilde c \in H_1(Q;\Z)$ maps to $c\in H_1(P;\Z)$.
\end{proposition}
\noindent
We will refer to $\gamma$ as the \emph{twisting class} of the embedding.
\begin{proof}
Consider the inclusion homomorphism $H_1(Q;\Z) \to H_1(V;\Z)$. The Mayer-Vietoris sequence (with integer coefficients) for $W=V \cup_Q P$ gives a short exact sequence
$$ 0 \to H_1(Q) \to H_1(V) \oplus H_1(P) \to H_1(W) \to 0.$$
Using the above splitting of $H_1(Q)$ we conclude that the torsion generator $\tilde c$ maps nontrivially into $H_1(W)$ (through $P$) and hence also into $H_1(V)$. It follows that $H_1(V)$ is an extension of $H_1(M)$ by $\Z_2[f]$.

Now use the long exact sequence of the triple $(V,\partial V, M_0)$, taking into account that $H^*(\partial V, M_0) \cong H^*(Q \sqcup M_1)$.
\begin{equation}\label{les}
\xymatrix{
H^2(V,\partial V) \ar[r] & H^2(V,M_0) \ar[r] & H^2(Q\sqcup M_1) \ar[r]^{\delta} \ar[d] & H^3(V,\partial V) \ar[d] \\
& & H_1(Q \sqcup M_1) \ar[r] & H_1(V)
}
\end{equation}
The vertical maps are given by the Poincar\'e-Lefschetz duality.  By the above the classes $\tilde c \in H_1(Q)$ and $[c]\in H_1(M_1)$ map nontrivially to $H_1(V)$ which by exactness implies the existence of $\gamma$.
\end{proof}

\subsection{A congruence for the normal Euler number}
Whitney showed~\cite{whitney:manifolds-lectures} that the normal Euler number of an embedded $F_h \subset \R^4$ is constrained by the congruence $e \equiv 2h \pmod4$; a similar congruence was given in higher dimensions by
Mahowald~\cite{mahowald:normal,massey:pontrjagin}.  We will make use of a similar congruence in deriving an embedding obstruction that involves the twisting element $\gamma \in H^2(V, M_0; \Z)$ from Proposition~\ref{twist}.   Denote by $\lk_{M}$ the linking form on the torsion subgroup of $H_1(M) = H^2(M)$.  In terms of cohomology, the linking of elements $x,\, y\in H^2(M)$ is given by $\lk_{M}(x,y) = \langle x \smile z,[M]\rangle \in \Q/\Z$ where $\delta z = y$ and $\delta$ is the Bockstein coboundary associated to the exact sequence $0 \to \Z \to \Q \to \Q/\Z \to 0$.

For an essential embedding of $F$ in a rational homology cobordism $W$, the twisting class $\gamma$ is an element of order $2$, so we can replace $\delta$ by the Bockstein $\beta$, writing $\gamma = \beta\tau$. Then the self-linking
\begin{equation}\label{E:tau-cup}
 \lk_{M}(\gamma,\gamma) =  \langle \gamma \smile \tau ,[M]\rangle\in (\frac12 \Z)/\Z \subset \Q/\Z
\end{equation}
is of the form $k_{[c]}/2$ where $k_{[c]} = 2\lk_{M}(\gamma,\gamma)  = 2\lk_{M}([c],[c])$ is either $0$ or $1$.
The following lemma is standard, and is proved using the naturality of the Bockstein $\beta$ and cup product.
\begin{lemma}\label{L:degree}
Let $M$ be an orientable $3$-manifold,  $b \in H^1(M;\Z_2)$, and let $a\in H^1(\RP^3;\Z_2)$ be the generator.  Then there is a map $\psi:M \to \RP^3$ such that $\psi^*a = b$.  Moreover, the degree of $\psi$ is given, modulo $2$, by $\langle b\smile \beta b, [M]\rangle$, and any degree satisfying this congruence is realized by some map $\psi$.
\end{lemma}

\begin{proposition}\label{P:mod4}
Let $W$ be a homology cobordism from $M_0$ to $M_1$ where $M_i$ is a rational homology sphere, and let $F_h \subset W$ be an essential embedding with normal Euler number $e$.  Then
\begin{equation}\label{E:mod4}
e \equiv 2k_{[c]}+ 2h\pmod4.
\end{equation}
\end{proposition}
\begin{proof}  We make use of an extension, due to B.-H. Li~\cite{li:whitney}, of the congruence of Whitney and Mahowald to the case of an embedding in an arbitrary oriented manifold.  In the case of an embedding $F_h \subset W$ of a surface in an orientable $4$-manifold $W$, it reads
\begin{equation}\label{E:li}
e \equiv \la\pont(\PD([F]),[W,\partial W]\ra + 2 w_1(\nu(F))^2  \pmod4.
\end{equation}
Here $\pont$ denotes the Pontrjagin square~\cite{pontrjagin:pi3}, a cohomology operation
$$
H^2(X,Y;\Z_2) \to H^4(X,Y;\Z_4)
$$
defined for a pair of spaces $(X,Y)$.  If $W$ is orientable, then $w_1(\nu(F)) = w_1(F)$, and there is the well-known relation $w_1(F)^2 + w_2(F) =0$.  Since $w_2(F)$ is the Euler characteristic mod $2$, and the Euler characteristic in turn is just $2-h$, we have that $2 w_1(\nu(F))^2  = 2h \pmod4$.
The other term takes a bit more work; we will compute it for $(W,\partial W)$ by comparison with the special case when $W = \RP^3 \times I$.

For any space $X$, the Pontrjagin square $\pont$ on $H^2(X\times I,X \times\{0,1\};\Z_2)$ is equivalent to another cohomology operation $P$, the Postnikov square~\cite{postnikov:square} defined on $H^1(X;\Z_2)$ via the following commutative diagram~\cite[Equation 5.5]{whitehead:obstructions}, where the vertical maps are isomorphisms coming from the long exact sequence of the triple $(X\times I,X \times\{0,1\},X\times\{0\})$:\footnote{We are grateful to Nikolai Saveliev for pointing out that this construction is given as an exercise in Postnikov's Russian edition~\cite{mosher-tangora:russian} of Mosher and Tangora's book on cohomology operations~\cite{mosher-tangora}.}
$$
\xymatrix{
H^2(X\times I,X \times\{0,1\};\Z_2) \ar[d]_{\cong} \ar[r]^{\pont} & H^4(X\times I,X \times\{0,1\};\Z_4)   \ar[d]_{\cong}\\
H^1(X;\Z_2) \ar[r]^{P} & H^3(X;\Z_4)
}
$$
It is possible, although tedious, to calculate the Postnikov square for $\RP^3$ directly in terms of a simplicial decomposition (the result is stated in~\cite{whitehead:obstructions} without proof), so we take an indirect but more efficient route. The embedding $\RP^2 \subset \RP^3 \times \frac12 \subset \RP^3 \times I$ has normal Euler number $0$, so Li's congruence together with the above relation between $\pont$ and $P$ implies that  $P(a) = 2\pmod4$, where $a\in H^1(\RP^3;\Z_2)$ is the generator.  Let $A \in H^2(\RP^3 \times I, \RP^3 \times \{0,1\});\Z_2)$ be the image of $a$ under the coboundary map of the long exact sequence of the pair $(\RP^3 \times I,  \RP^3 \times \{0,1\})$.  It follows that $\pont(A)$ is the element  $2 \in H^4(\RP^3 \times I,  \RP^3 \times \{0,1\});\Z_4)\cong \Z_4$.

Now we turn to the evaluation of $\pont$ on $W$.
Because $W$ is a homology cobordism, there is a unique class $(x_0,x_1) \in H^1(M_0 \sqcup M_1;\Z_2) \cong H^1(M_0;\Z_2) \oplus H^1(M_1;\Z_2)$ such that $(x_0,0)$ and $(0,x_1)$ are mapped to $\PD_W([F]) $ under the coboundary map in the long exact sequence of the pair $(W, M_0 \sqcup M_1)$.  By Lemma~\ref{L:degree}, the class $x_i\in H^1(M_i;\Z_2)$ produces a map $\psi_i\co M_i\to \RP^3$, and it is straightforward to see that these maps have the same mod $2$ degree, given by  $\langle x_i\smile \beta x_i, [M_i]\rangle$.  Hence we may modify one of them so that $\degree(\psi_0) = \degree(\psi_1)$. A simple obstruction theory argument produces a  map
$$
\Psi\co (W,M_0,M_1) \to (\RP^3 \times I,RP^3 \times \{0\}, RP^3 \times \{1\})
$$
extending the $\psi_i$, and with $\degree(\Psi) = \degree(\psi_i)$.

Since, by construction, $\psi_0^*(a) = x_0$, it follows that $\Psi^*(A) =\PD_W([F])$, and hence that $\pont(\PD_W([F])) = \degree(\Psi)  \pont(A) = 2 \degree(\psi_0) \pmod{4}$.  Equation~\eqref{E:tau-cup} and naturality of the cup product and Bockstein imply that $\psi_0$ has degree congruent to $k_{[c]}$ mod $2$.  So  $\pont(\PD_W([F])) \equiv 2k_{[c]} \pmod4$ and the proposition follows from~\eqref{E:li}.
\end{proof}

\section{Heegaard Floer correction terms for manifolds with \texorpdfstring{$b_1>0$}{b1>0}}\label{S:correction}

In this section, we review some facts about the Ozsv\'ath--Szab\'o correction terms for $3$-manifolds with positive first Betti number. Many of these results are straightforward generalizations of the corresponding results for rational homology spheres given in \cite{oz:boundary} and are familiar to experts.

Let $Y$ be a closed, oriented $3$-manifold. We write $H_1^T(Y)$ for $H_1(Y;\Z)/\Tors$. Note that $H_1^T(Y)$ and $H^1(Y)$ are canonically dual to one another. Thus, the exterior algebra $\Lambda^* H_1^T(Y)$ acts canonically on $\Lambda^* H^1(Y)$, taking
\[
\Lambda^k H_1^T(Y) \otimes \Lambda^\ell H^1(Y) \to \Lambda^{\ell-k} H^1(Y).
\]
Note that the kernel of this action (i.e., the set of elements of $\Lambda^* H^1(Y)$ annihilated by all of $H_1^T(Y)$) is the bottom exterior power $\Lambda^0 H^1(Y) \cong \Z$, while the top exterior power $\Lambda^{b_1(Y)} H^1(Y)$ maps isomorphically to the cokernel of the action (i.e., $\Lambda^* H^1(Y) / (H_1^T(Y) \cdot \Lambda^* H^1(Y))$). Furthermore, the action satisfies the following useful property:
If $\gamma_1, \dots, \gamma_k$ are elements of a basis for $H_1^T(Y)$, and $\omega \in \Lambda^\ell H^1(Y)$ is an element such that $(\gamma_1 \wedge \cdots \wedge \gamma_k) \cdot \omega = 0$, then there exists $\omega' \in \Lambda^{\ell + k}$ such that $(\gamma_1 \wedge \cdots \wedge \gamma_k) \cdot \omega'= \omega$.

\begin{definition} \label{def:standardHFi}
Let $Y$ be a closed, oriented $3$-manifold. We say that $\HFi(Y)$ is \emph{standard} if for each torsion spin$^c$ structure $\spinct$ (i.e. with torsion $c_1(\spinct)$) on $Y$, we have
\begin{equation} \label{eq:standardHFi}
\HFi(Y,\spinct) \cong \Lambda^* H^1(Y;\Z) \otimes \Z[U,U^{-1}]
\end{equation}
as a relatively graded $\Lambda^*H_1^T(Y) \otimes \Z[U,U^{-1}]$--module.
\end{definition}

It follows from the discussion above that when $\HFi(Y)$ is standard, the kernel and cokernel of the action of $\Lambda^* H_1^T(Y)$ on $\HFi(Y, \spinct)$ are each isomorphic to $\Z[U,U^{-1}]$.

\begin{theorem} \label{thm:Lidman}
If $Y$ is a closed oriented $3$-manifold such that the triple cup product map
\[
H^1(Y;\Z) \otimes H^1(Y;\Z) \otimes H^1(Y;\Z) \to \Z
\]
given by
\[
\alpha \otimes \beta \otimes \gamma \mapsto \gen{\alpha \smile \beta \smile \gamma, [Y]}
\]
vanishes identically, then $\HFi(Y)$ is standard.
\end{theorem}

\begin{proof}
This is essentially a result of Lidman~\cite{lidman:hfinfinity}; the one thing to note is that the isomorphisms described in Lidman's paper all respect the $H_1$ action.
\end{proof}

\begin{definition}
Let $Y$ be a closed, oriented $3$-manifold with standard $\HFi$, and let $\spinct$ be a torsion spin$^c$ structure on $Y$. The \emph{bottom correction term} (or \emph{bottom $d$-invariant}) $\dbot(Y, \spinct)$ is the minimal grading in which the restriction of the map $\pi\co \HFi(Y, \spinct) \to \HFp(Y, \spinct)$ to the kernel of the $H_1$ action on $\HFi(Y, \spinct)$ is nontrivial. The \emph{top correction term} (or \emph{top $d$-invariant}) $\dtop(Y, \spinct)$ is the minimal grading in which the induced map
\[
\bar\pi\co \HFi(Y, \spinct)/ (H_1^T(Y) \cdot \HFi(Y, \spinct))  \to \im (\pi)/ (H_1^T(Y) \cdot \im (\pi))
\]
is nontrivial.
\end{definition}

Note that when $Y$ is a rational homology sphere, the $H_1$ action is trivial, so we have $\dbot(Y,\spinct) = \dtop(Y, \spinct) = d(Y, \spinct)$.

\begin{example} \label{ex:S1xS2}
For any $n \ge 0$, consider the manifold $\conn^n S^1 \times S^2$, with its unique torsion spin$^c$ structure $\spincs_0$. As shown by Ozsv\'ath and Szab\'o, the group $\HF^{\le 0}(\conn^n S^1 \times S^2, \spincs_0) \subset \HFi(\conn^n S^1 \times S^2, \spincs_0)$ has a canonical top-dimensional generator (up to sign), which we denote by $\Theta_n^\top$. Also, let $\Delta$ be a generator of $\Lambda^n H_1(\conn^n S^1 \times S^2;\Z)$, and let $\Theta_n^\bot = \Delta \cdot \Theta_n^\top$, again defined up to sign. Note that $\gr (\Theta_n^\top) = n/2$ and $\gr (\Theta_n^\bot) = -n/2$. It is well-known that $\pi(\Theta_n^\top)$ and $\pi(\Theta_n^\bot)$ are both nonzero in $\HFp(\conn^n S^1 \times S^2, \spincs_0)$ and are in the kernel of $U$. Furthermore, $\pi(\Theta_n^\top)$ survives in the cokernel of the $H_1$ action on $\HFp$, and $\pi(\Theta_n^\bot)$ is in the kernel of the $H_1$ action. Thus,
\[
\dtop(\conn^n S^1 \times S^2, \spincs_0) = n/2 \quad \text{and} \quad \dbot(\conn^n S^1 \times S^2, \spincs_0) = -n/2.
\]
\end{example}

\begin{lemma} \label{lemma:topbot}
For any closed, oriented $3$-manifold $Y$ with standard $\HFi$ and any torsion spin$^c$ structure $\spinct$ on $Y$, we have
\[
\dbot(Y, \spinct) \ge \dtop(Y,\spinct) - b_1(Y) \quad \text{and} \quad \dbot(Y, \spinct) \equiv \dtop(Y,\spinct) - b_1(Y) \pmod {2\Z}.
\]
\end{lemma}

\begin{proof}
Write $d = \dbot(Y, \spinct)$, and let $\xi \in \HFi_{d}(Y, \spinct)$ be an element of the kernel of the $H_1$ action whose image $\pi(\xi) \in \HFp(Y, \spinct)$ is nonzero. Choose a basis $\gamma_1, \dots, \gamma_n$ for $H_1^T(Y)$, where $n = b_1(Y)$. Because $\HFi$ is standard, we may find $\eta \in \HFi_{d+n}(Y, \spincs)$ such that $(\gamma_1 \wedge \cdots \wedge \gamma_n) \cdot \eta = \xi$. Then the class of $\pi(\eta)$ modulo the $H_1$ action is nonzero, so
\[
\dtop(Y, \spincs) \le \gr\pi(\eta) = \gr\xi + n = \dbot(Y, \spincs) + n.
\]
Moreover, since $\HFi(Y, \spinct) / (H_1^T(Y) \cdot \HFi(Y, \spinct)) \cong \Z[U,U^{-1}]$, any other nonzero element $\eta'$ of this module must have $\gr(\eta') \equiv \gr(\eta) \pmod 2$, which gives the second statement.
\end{proof}

\begin{proposition}[Conjugation invariance] \label{prop:conjugation}
Let $Y$ be a closed, oriented $3$-manifold with standard $\HFi$, let $\spincs$ be a torsion spin$^c$ structure on $Y$, and let $\bar \spincs$ denote the conjugate spin$^c$ structure. Then $\dbot(Y, \spincs) = \dbot(Y, \bar\spincs)$ and $\dtop(Y, \spincs) = \dtop(Y, \bar\spincs)$.
\end{proposition}

\begin{proof}
This follows immediately from the conjugation invariance of Heegaard Floer homology.
\end{proof}

\begin{proposition}[Duality] \label{prop:duality}
Let $Y$ be a closed, oriented $3$-manifold with standard $\HFi$, and let $\spincs$ be a torsion spin$^c$ structure on $Y$. Then $\dbot(Y, \spincs) = -\dtop(-Y, \spincs)$ and $\dtop(Y, \spincs) = -\dbot(-Y, \spincs)$.
\end{proposition}

\begin{proof}
This is a straightforward adaptation of \cite[Proposition 4.2]{oz:boundary}, making use of the isomorphism between the Heegaard Floer homology of $Y$ and the Heegaard Floer cohomology of $-Y$, along with the fact that dualizing interchanges the roles of the kernel and cokernel of the $H_1$ action.
\end{proof}

\begin{proposition}[Additivity] \label{prop:additivity}
Let $Y$ and $Z$ be closed, oriented $3$-manifolds with standard $\HFi$, and let $\spinct$ and $\spincu$ be torsion spin$^c$ structures on $Y$ and $Z$ respectively. Then
\[
\dbot(Y \conn Z, \spinct \conn \spincu) = \dbot(Y, \spinct) + \dbot(Z, \spincu)
\]
and
\[
\dtop(Y \conn Z, \spinct \conn \spincu) = \dtop(Y, \spinct) + \dtop(Z, \spincu).
\]
\end{proposition}

\begin{proof}
Consider the map
\[
\iota^Y \co \HF^{\le0}(Y, \spinct) \to \HFi(Y, \spinct).
\]
It is not hard to show that $\dbot(Y, \spinct)$ is equal to the maximal degree in which the restriction of $\iota^Y$ to the kernel of the $H_1$ action is nonzero, and $\dtop(Y, \spinct)$ is the maximal degree in which the map on cokernels induced by $\iota^Y$ is nonzero. By the connected sum formula for Heegaard Floer homology, there are graded isomorphisms making the diagram
\[
\xymatrix{
H_*(\CF^{\le0}(Y, \spinct) \otimes_{\Z[U]} \CF^{\le0} (Z, \spincu)) \ar[r]^-{F^-_{Y \conn Z}}_-{\cong} \ar[d]^{(\iota^Y \otimes \iota^Z)_*}  &  \HF^{\le0}(Y \conn Z, \spinct \conn \spincu) \ar[d]^{\iota^{Y \conn Z}} \\
H_*(\CF^\infty(Y, \spinct) \otimes_{\Z[U,U^{-1}]} \CF^\infty (Z, \spincu)) \ar[r]^-{F^\infty_{Y \conn Z}}_-{\cong} &
  \HF^\infty (Y \conn Z)
}
\]
commute. Furthermore, identifying $\Lambda^* H_1^T(Y \conn Z)$ with $\Lambda^* H_1^T(Y) \otimes \Lambda^* H_1^T(Z)$, the horizontal maps respect the $H_1$ action. Combining this result with the algebraic K\"unneth theorem and the fact that $Y$, $Z$, and $Y \conn Z$ have standard $\HFi$, we have a diagram
\[
\xymatrix{
0 \ar[r] & \HF^{\le0}(Y, \spinct) \otimes_{\Z[U]} \HF^{\le0}(Z, \spincu) \ar[r]^-{F^{\le0}_{Y \conn Z}} \ar[d]^{\iota^Y \otimes \iota^Z} & \HF^{\le0}(Y \conn Z, \spinct \conn \spincu) \ar[r] \ar[d]^{\iota^{Y \conn Z}} & T \ar[r] & 0 \\
& \HF^\infty(Y, \spinct) \otimes_{\Z[U]} \HF^\infty(Z, \spincu) \ar[r]^-{F^\infty_{Y \conn Z}}_-{\cong} & \HF^\infty({Y \conn Z}, {\spinct \conn \spincu}) & & }
\]
where $T$ is a torsion $\Z[U]$--module and the top row is exact. Some diagram-chasing then shows that
\[
\dbot(Y \conn Z, \spinct \conn \spincu) \le \dbot(Y, \spinct) + \dbot(Z, \spincu)
\]
and
\[
\dtop(Y \conn Z, \spinct \conn \spincu) \le \dtop(Y, \spinct) + \dtop(Z, \spincu).
\]
Applying the same reasoning to $-(Y \conn Z)$, we see that
\[
\dbot(-(Y \conn Z), \spinct \conn \spincu) \le \dbot(-Y, \spinct) + \dbot(-Z, \spincu)
\]
and
\[
\dtop(-(Y \conn Z), \spinct \conn \spincu) \le \dtop(-Y, \spinct) + \dtop(-Z, \spincu).
\]
The desired result then follows from Proposition \ref{prop:duality}.
\end{proof}

The key property of the $\dbot$ and $\dtop$ invariants is their behavior with respect to negative semidefinite $4$-manifolds bounding a given $3$-manifold:

\begin{theorem}[cf.~{\cite[Theorem 9.15]{oz:boundary}}] \label{T:negative}
Let $Y$ be a closed, oriented $3$-manifold with standard $\HFi$, equipped with a torsion spin$^c$ structure $\spinct$. If $X$ is a negative semidefinite $4$-manifold bounded by $Y$ such that the restriction map $H^1(X;\Z) \to H^1(Y;\Z)$ is trivial, then for each spin$^c$ structure $\spincs$ on $X$ that restricts to $\spinct$, we have
\[
c_1(\spincs)^2 + b_2^-(X) + 4b_1(X) \le 4\dbot(Y,\spinct) + 2b_1(Y).
\]
Moreover, the two sides of this inequality are congruent modulo $8$.
\end{theorem}

\begin{proof}
Let $X' = X \minus B^4$, and view $X'$ as a cobordism from $S^3$ to $Y$. Consider the commutative diagram
\[
\xymatrix{
\HFi(S^3)  \ar[r]^{F^\infty_{X', \spincs}} \ar[d]^{\pi^{S^3}}  & \HFi(Y, \spinct) \ar[d]^{\pi^{Y}} \\
\HFp(S^3)  \ar[r]^{F^+_{X', \spincs}} & \HFp(Y, \spinct).
}
\]
Using the fact that $X'$ is negative semidefinite, the argument in \cite[Section 9]{oz:boundary} can be generalized to show that $F^\infty_{X', \spincs}$ is injective, with image equal to the kernel of the $H_1$ action on $\HFi(Y, \spinct)$. Furthermore, the horizontal maps shift the grading by
\[
D = \frac{c_1(\spincs)^2 - 2\chi(X') -3\sigma(X')}{4} = \frac{c_1(\spincs)^2 + 4b_1(X) + b_2^-(X) - 2b_1(Y)}{4},
\]
since $\chi(X') = -2b_1(X) + b_2(X) = -2b_1(X) + b_2^-(X) + b_1(Y)$ and $\sigma(X') = -b_2^-(X)$. Let $d = \dbot(Y, \spinct)$, and let $\xi \in \HFi_d(Y, \spinct)$ be an element in the kernel of the $H_1$ action such that $\pi^Y(\xi) \ne 0$. There exists some $\eta \in \HFi_{d-D}(S^3)$ such that $F^\infty_{X', \spincs}(\eta) = \xi$. Commutativity of the diagram shows that $\pi^{S^3}(\eta) \ne 0$, so $d-D \ge 0$ and $d - D \equiv 0 \pmod 2$. The result then follows.
\end{proof}

Likewise, we have:

\begin{theorem} \label{T:QHcob}
Let $Y_0$ and $Y_1$ be closed, oriented manifolds with standard $\HFi$, equipped with torsion spin$^c$ structures $\spinct_0$ and $\spinct_1$, respectively. If $W$ is a negative-semidefinite cobordism from $Y_0$ to $Y_1$ such that the restriction maps $H^1(W) \to H^1(Y_0)$ and $H^1(W) \to H^1(Y_1)$ are isomorphisms, and $\spincs$ is a spin$^c$ structure on $W$ that restricts to $\spinct_0$ on $Y_0$ and to $\spinct_1$ on $Y_1$, then
\begin{align*}
\dbot(Y_0, \spinct_0) \le \dbot(Y_1, \spinct_1) - \frac{c_1(\spincs)^2 + b_2^-(W)}{4} \\
\intertext{and}
\dtop(Y_0, \spinct_0) \le \dtop(Y_1, \spinct_1) - \frac{c_1(\spincs)^2 + b_2^-(W)}{4}.
\end{align*}
In particular, if $(W, \spincs)$ is a spin$^c$ rational homology cobordism, then
\[
\dbot(Y_0, \spinct_0) = \dbot(Y_1, \spinct_1)  \quad \text{and} \quad
\dtop(Y_0, \spinct_0) = \dtop(Y_1, \spinct_1).
\]
\end{theorem}

\begin{proof}
This follows from the usual argument and the observation (using \cite[Section 9]{oz:boundary}) that
\[
F^\infty_{W, \spincs}\co \HFi(Y_0, \spinct_0) \to \HFi(Y_1, \spinct_1)
\]
is an isomorphism that respects the $H_1$ action.
\end{proof}

\begin{corollary} \label{C:QHS1xB3}
Let $Y$ be a closed, oriented $3$-manifold with standard $\HFi$, equipped with a torsion spin$^c$ structure $\spinct$, and let $b= b_1(Y)$. If $X$ is an oriented $4$-manifold bounded by $Y$ such that $b_1(X) = b$ and $b_2(X) = b_3(X)=0$, and $\spinct$ extends over $W$, then $\dbot(Y, \spinct) = -b/2$ and $\dtop(Y, \spinct) = b/2$.
\end{corollary}

\begin{proof}
Deleting a regular neighborhood of a bouquet of circles representing a basis for $H_1(X;\Q)$ yields a spin$^c$ rational homology cobordism from $(\conn^b S^1 \times S^2, \spincs_0)$ to $(Y, \spinct)$; apply Theorem \ref{T:QHcob} and Example \ref{ex:S1xS2}.
\end{proof}

\section{Non-orientable genus bounds from the correction terms}\label{S:bound}

Let $W$ be a homology cobordism between rational homology spheres $M_0$ and $M_1$, and suppose $F=F_h \subset W$ is an essential non-orientable surface of genus $h$ with normal Euler number $e$. Recall that we denote the normal disk bundle of $F$ in $W$ by $P=P_{h,e}$ and its boundary by $Q=Q_{h,e}$. We orient $Q$ as the boundary of $P$. Let $V = W \smallsetminus \operatorname{int} P$; with this convention, $\partial V = -M_0 \sqcup -Q \sqcup M_1$.

For any $\spincs \in \Spin^c(M_0)$ we denote by $\spincs$ the unique extension of $\spincs$ to $W$. Let $\tilde\spincs \in \Spin^c(V)$ be the unique spin$^c$ structure that restricts to $\spincs$ on $M_0$ and does not extend over $W$; i.e., $\tspincs=\spincs+\gamma$, where $\gamma$ is the twisting element from Proposition \ref{twist}. Let $\t_\spincs \in \Spin^c(Q)$ be the restriction of $\tilde \spincs$ to $Q$; this is one of the two torsion spin$^c$ structures on $Q$ that do not extend over the disk bundle $P$. The restriction of $\tilde\spincs$ to $M_1$ is $\spincs + \phi$, where $\phi=\PD([c])$ and $c \in H_1(F;\Z)$ is the torsion generator.

\begin{theorem}\label{T:twist-bound}
Let $W$ be a homology cobordism between rational homology spheres $M_0$ and $M_1$ and suppose that $F_h \subset W$ is essential and has normal Euler number $e$. Then for each $\s \in \Spin^c(M_0)$, we have
\begin{equation} \label{dbounds}
\dtop(Q_{h,e},\t_s) - \frac{h-1}{2} \le d(M_1,\s+\phi)) - d(M_0, \s) \le \dbot(Q_{h,e}, \t_\s) + \frac{h-1}{2}.
\end{equation}
Moreover, the three quantities in \eqref{dbounds} are congruent modulo $2$.
\end{theorem}

\begin{proof}
Let $V'$ be obtained from $V$ by deleting neighborhoods of an arc connecting $M_0$ and $Q$ and an arc connecting $Q$ to $M_1$; thus, $\partial V' = -M_0 \conn -Q \conn M_1$. We denote the restriction of $\tilde\spincs$ to $V'$ by $\tilde\spincs$. Note that the intersection form on $H_2(V')$ is zero, since all of $H_2(V')$ comes from the boundary, hence $c_1(\tilde\spincs)^2=0$. Furthermore, it is easy to see (cf.~the proof of Proposition~\ref{twist}) that $b_1(V) = b_1(V')=0$ and $b_2(V) = b_2(V')=h-1$.

Applying Theorem \ref{T:negative} to $(V', \tilde\s)$, we have:
\[
\begin{aligned}
0 &\le 4 \dbot(\partial V', \tilde\spincs) + 2b_1(\partial V) \\
&= 4(d(-M_0, \spincs) + \dbot(-Q, \spinct_\spincs) + d(M_1, \spincs+\phi)) + 2(h-1),
\end{aligned}
\]
so
\[
\begin{aligned}
d(M_1, \spincs+\phi) - d(M_0, \spincs) &\ge -\dbot(-Q, \spinct_\spincs) - \frac{h-1}{2} \\
&= \dtop(Q, \spinct_\spincs) - \frac{h-1}{2}
\end{aligned}
\]
using Proposition \ref{prop:duality}. Likewise, applying Theorem \ref{T:negative} to $(-V', \tilde\s)$, we have:
\[
\begin{aligned}
0 &\le 4 \dbot(\partial (-V'), \tilde\spincs) + 2b_1(\partial V) \\
&= 4(d(M_0, \spincs) + \dbot(Q, \spinct_\spincs) + d(-M_1, \spincs+\phi)) + 2(h-1).
\end{aligned}
\]
so
\[
d(M_1, \spincs+\phi) - d(M_0, \spincs) \le \dbot(Q, \spinct_\spincs) + \frac{h-1}{2}.
\]
Moreover, in each of these inequalities, the two sides are congruent modulo $2$.
\end{proof}

\section{Circle bundles over non-orientable surfaces}\label{S:floercalc}
Let $P_{h,e}$ and $Q_{h,e}$ denote the orientable disk bundle and circle bundle over $F_h=\#^h \RP^2$ of Euler number $e$, respectively. We orient $Q_{h,e}$ as the boundary of $P_{h,e}$. The goal of this section is to determine the correction terms $\dbot$ and $\dtop$ for the torsion spin$^c$ structures on $Q_{h,e}$.

Note that the torsion spin$^c$ structures on $Q_{h,e}$ are in one-to-one correspondence with the torsion part of $H^2(Q_{h,e})$, which by Lemma \ref{L:Q} is isomorphic to $\Z_2 \oplus \Z_2$ when $e$ is even and $\Z_4$ when $e$ is odd. In either case, two of the torsion spin$^c$ structures on $Q_{h,e}$ extend over $P_{h,e}$ and two do not. Using a precise notational convention established below, we label the two extendible spin$^c$ structures by $\spincu^{h,e}_0, \spincu^{h,e}_1$ and the two non-extendible ones by $\spinct^{h,e}_0, \spinct^{h,e}_1$. (We omit $h$ and $e$ when they are clear from the context.) The rest of this section is devoted to the proof of the following theorem:

\begin{theorem} \label{T:d}
Let $h>0$ and $e\in\Z$. For $a=0,1$, the correction terms of $Q_{h,e}$ in the non-extendible spin$^c$ structures are given by
\[
\dbot(Q_{h,e},\spinct_a^{h,e})=\dtop(Q_{h,e},\spinct_a^{h,e})= \frac{e-2}{4} +a
\]
for $h$ odd, and by
\[
\dbot(Q_{h,e},\spinct_a^{h,e})=\dtop(Q_{h,e},\spinct_{1-a}^{h,e})=\frac{e-2}{4} + a
\]
for $h$ even. The correction terms in the extendible spin$^c$ structures are given by
\[
\dbot(Q_{h,e}, \spincu_a^{h,e}) = - \frac{h-1}{2} \quad \text{and} \quad \dtop(Q_{h,e}, \spincu_a^{h,e}) = \frac{h-1}{2}.
\]
\end{theorem}

We may obtain $Q_{h,e}$ as surgery on a knot in $\#^h S^1 \times S^2$, namely the connected sum of $h$ copies of a knot in $S^1 \times S^2$ representing twice the generator of first homology. By \cite[Proposition 9.3]{oz:boundary}, we see that $Q_{h,e}$ has standard $\HFi$; thus, the aforementioned correction terms are actually defined. Let $R_{h,e}$ denote the $2$-handle cobordism corresponding to the surgery; note that $P_{h,e} = (\natural^h S^1 \times B^3) \cup R_{h,e}$.

\begin{figure}[h]
  \begin{center}
    \labellist
    \normalsize\hair 0mm
    \pinlabel {{\footnotesize$0$}} at 19 87
    \pinlabel {{\footnotesize$0$}} at 51 105
    \pinlabel {{\footnotesize$0$}} at 97 87
    \pinlabel {{\footnotesize$0$}} at 119 105
    \pinlabel {{\footnotesize$0$}} at 232 87
    \pinlabel {{\footnotesize$0$}} at 268 105
    \pinlabel {$=\ \#_g$} at 193 32
    \endlabellist
    \includegraphics[scale=.7]{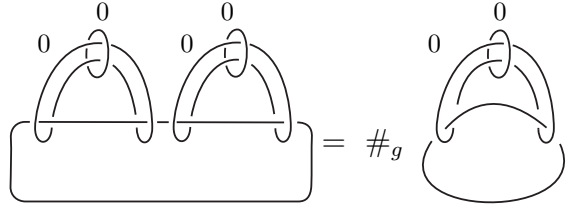}
  \end{center}
  \caption{The Borromean knot.}
  \label{F:borromean}
\end{figure}

Denote by $M_{g,e}$ the oriented circle bundle of Euler number $e$ over an orientable surface of genus $g$. (Note that $M_{0,e}$ is the lens space $L(e,1)$.) The manifold $M_{g,e}$ is obtained by surgery with coefficient $e$ on the ``Borromean'' knot $B_g \subset \#^{2g} S^1 \times S^2$, drawn on the left in Figure~\ref{F:borromean}; let $V_{g,e}$ denote the corresponding $2$-handle cobordism from $\#^{2g} S^1 \times S^2$ to $M_{g,e}$.

Let $\Sigma \subset \#^{2g} S^1 \times S^2$ be a genus-$g$ Seifert surface for $B_g$, which can be capped off to give a closed surface $\hat \Sigma_e \subset V_{g,e}$ with $[\hat \Sigma_e]^2=e$. For $i \in \Z$, let $\tilde \spincs^{g,e}_i \in \Spin^c(V_{g,e})$ denote the spin$^c$ structure that restricts to the unique torsion spin$^c$ structure on $\#^{2g} S^1 \times S^2$ and satisfies
\[
\gen{c_1(\tilde \spincs^{g,e}_i), [\hat \Sigma_e]} + e = 2i.
\]
Let $\spincs^{g,e}_i$ be the restriction of $\tilde \spincs^{g,e}_i$ to $M_{g,e}$; this depends only on the class of $i$ modulo $e$. (We may omit $g$ and $e$ when they are understood from context.)

When $h = 2g+1$, the non-orientable surface $F_h$ can be viewed as a connected sum of $g$ tori and a copy of $\RP^2$. Similarly, when $h=2g+2$, $F_h$ is a connected sum of $g$ tori and two copies of $\RP^2$. Therefore, the surgery diagram in Figure \ref{F:bothgenus} represents $Q_{2g+2,e}$, whereas the same diagram without the rightmost $0$-framed surgery curve represents $Q_{2g+1,e+2}$.

\begin{figure}[h]
  \begin{center}
    \labellist
    \normalsize\hair 0mm
    \pinlabel {{\footnotesize$0$}} at 22 70
    \pinlabel {{\footnotesize$0$}} at 54 90
    \pinlabel {{\footnotesize$0$}} at 105 70
    \pinlabel {{\footnotesize$0$}} at 135 90
    \pinlabel {{\footnotesize$0$}} at 218 90
    \pinlabel {{\footnotesize$0$}} at 285 90
    \pinlabel {{\footnotesize$e$}} at 338 28
    \endlabellist
    \includegraphics[scale=.8]{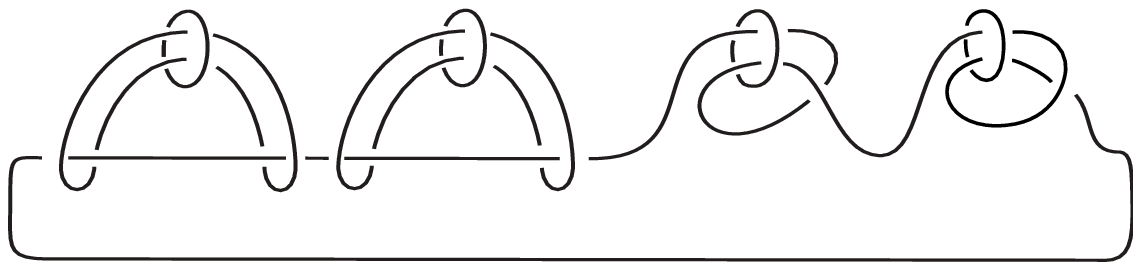}
  \end{center}
  \caption{Surgery description of $Q_{h,e}$ for $h$ even; omitting the right-most zero framed surgery curve gives $Q_{h-1,e+2}$.}
  \label{F:bothgenus}
\end{figure}

We set up some notation used in the proof of the theorem. Let $X=X_{g,e}$ be the cobordism from $\#^{2g} S^1 \times S^2$ obtained by attaching five $2$-handles $h_0,\ldots, h_4$, as shown in Figure~\ref{F:Xgn}.
\begin{figure}[h]
  \begin{center}
    \labellist
    \normalsize\hair 0mm
    \pinlabel {{\footnotesize$\langle0\rangle$}} at 52 107
    \pinlabel {{\footnotesize$\langle0\rangle$}} at 80 125
    \pinlabel {{\footnotesize$\langle0\rangle$}} at 130 107
    \pinlabel {{\footnotesize$\langle0\rangle$}} at 155 125
    \pinlabel {{\footnotesize$e$}} at 323 47
    \pinlabel {{\footnotesize$h_0$}} at 302 48
    \pinlabel {{\footnotesize$h_1$}} at 195 115
    \pinlabel {{\footnotesize$0$}} at 215 123
    \pinlabel {{\footnotesize$0$}} at 280 123
    \pinlabel {{\footnotesize$0$}} at 248 115
    \pinlabel {{\footnotesize$h_2$}} at 302 115
    \pinlabel {{\footnotesize$-1$}} at -3 49
    \pinlabel {{\footnotesize$h_4$}} at 42 48
    \pinlabel {{\footnotesize$h_3$}} at 222 102
    \endlabellist
    \includegraphics[scale=.9]{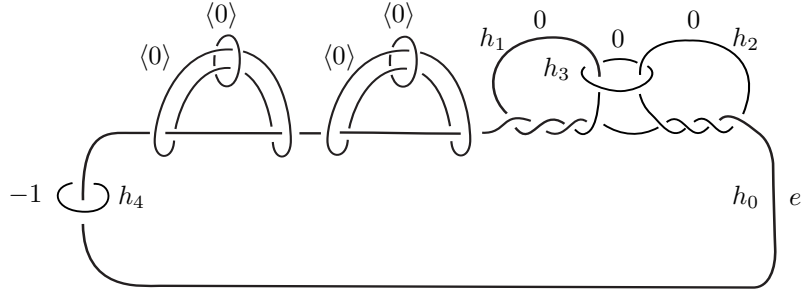}
  \end{center}
  \caption{The composite cobordism $X_{g,e}$.}
  \label{F:Xgn}
\end{figure}
We will use various subcobordisms of $X$ to relate the $d$-invariants of different 3-manifolds. The interesting part, where $h_1$ was already attached, is shown in the following diagram:
\begin{equation} \label{eq:handles}
\xymatrix @C=0.9in @R=0.5in{
\#^{2g+1} S^1 \times S^2 \ar[d]^{R_{2g+1,e+2}}_{h_0} \ar[r]^{}_{h_2} & \#^{2g+2} S^1 \times S^2 \ar[r]_{h_3}^{}  \ar[d]_{h_0}^{R_{2g+2,e}} & \#^{2g+1} S^1 \times S^2 \ar[d]_{h_0}^{V_{g,e} \bconn\, \mathbb{I}} \\
Q_{2g+1,e+2} \ar[r]^{W_2}_{h_2} \ar[d]^{W_4}_{h_4} & Q_{2g+2,e} \ar[r]^{W_3}_{h_3} \ar[d]^{W_4'}_{h_4} & M_{g,e} \conn S^1 \times S^2 \\
 Q_{2g+1,e+3} \ar[r]^{}_{h_2} & Q_{2g+2,e+1}
}
\end{equation}
Here, $\mathbb{I}$ denotes the product cobordism $(S^1 \times S^2) \times I$.

Note that all of the 2-handles in $X$ are attached to $\#^{2g} S^1 \times S^2$ along nullhomologous curves. Thus, $H_2(X) \cong H_2(\#^{2g} S^1 \times S^2) \oplus \Z^5$, where the generators $A_\ell$, $\ell=0,\ldots,4$ of $\Z^5$ represent classes given by $h_\ell$ with the exception of $A_2$, which represents the sum of $h_1$ and $h_2$ (corresponding to sliding $h_2$ across $h_1$). This generator is represented by a sphere of square zero and has vanishing intersection number with the other generators. The intersection form of $X$ relative to $A_0, A_1, A_3$ and $A_4$ is given by
\[
\begin{bmatrix}
e & -2 & 0 & 1\\ -2 & 0 & 1 & 0 \\ 0 & 1 & 0 & 0 \\ 1 & 0 & 0 & -1
\end{bmatrix}
\]

A spin$^c$ structure on $X$ is specified by its restriction to $\#^{2g} S^1 \times S^2$ and the evaluation of its first Chern class on the $A_\ell$. Specifically, for integers $i, j, k, l$, let $\spincv_{i,j,k,l} \in \Spin^c(X)$ be the spin$^c$ structure that restricts to the unique torsion spin$^c$ structure on $\#^{2g} S^1 \times S^2$ and satisfies:
\[
\begin{aligned}
\gen{c_1(\spincv_{i,j,k,l}), A_0} &+ e = 2i \\
\gen{c_1(\spincv_{i,j,k,l}), A_1} & = 2j \\
\gen{c_1(\spincv_{i,j,k,l}), A_2} & = 0 \\
\gen{c_1(\spincv_{i,j,k,l}), A_3} & = 2k \\
\gen{c_1(\spincv_{i,j,k,l}), A_4} &-1 = 2l.
\end{aligned}
\]

By considering the presentations for $H^2$ of the various $3$-manifolds given by the intersection forms of
subcobordisms of $X$, we observe the following facts:

\begin{itemize}
\item
The restriction of $\spincv_{i,j,k,l}$ to $\#^{2g+1} S^1 \times S^2$ is torsion iff $j=0$.

\item
If $e$ is even, then $\spincv_{i,j,k,l} |_{Q_{2g+1, e+2}} = \spincv_{i',j',k',l'}|_{Q_{2g+1, e+2}}$ iff $i \equiv i' \pmod 2$ and $j \equiv j' \pmod 2$. If $e$ is odd, then $\spincv_{i,j,k,l} |_{Q_{2g+1, e+2}} = \spincv_{i',j',k',l'}|_{Q_{2g+1, e+2}}$ iff $2i-j  \equiv 2i'-j' \pmod 4$. The same is true for the restrictions to $Q_{2g+2,e}$.

\item
If $e$ is even, then $\spincv_{i,j,k,l} |_{Q_{2g+1, e+3}} = \spincv_{i',j',k',l'}|_{Q_{2g+1, e+3}}$ iff $2i-j+2l \equiv 2i'-j'+2l' \pmod 4$. If $e$ is odd, then $\spincv_{i,j,k,l} |_{Q_{2g+1, e+3}} = \spincv_{i',j',k',l'}|_{Q_{2g+1, e+3}}$ iff $j \equiv j' \pmod 2$ and $i+l \equiv i'+l' \pmod 2$. The same is true for the restrictions to $Q_{2g+2, e+1}$.
\end{itemize}


Thus, for each $(h,f) \in \{(2g+1,e+2), (2g+1,e+3), (2g+2,e), (2g+2,e+1)\}$, we pin down the labeling of the four torsion spin$^c$ structures on $Q_{h,e}$ by setting
\[
\begin{aligned}
\spincu_0^{h,f} &= \spincv_{g+e,0,0,0} |_{Q_{h,f}} & \spincu_1^{h,f} &= \spincv_{g+e+1,0,0,0} |_{Q_{h,f}} \\
\spinct_0^{h,f} &= \spincv_{g+e,1,0,0} |_{Q_{h,f}} & \spinct_1^{h,f} &= \spincv_{g+e+1,1,0,0} |_{Q_{h,f}}.
\end{aligned}
\]
Note that $\spincu_0^{h,f}$ and $\spincu_1^{h,f}$ extend over the disk bundle $P_{h,f}$, while $\spinct_0^{h,f}$ and $\spinct_1^{h,f}$ do not.

\begin{remark}
We must verify that this labeling of the spin$^c$ structures on $Q_{2g+1,e+3}$ and $Q_{2g+2,e+2}$ is consistent with the labeling obtained by considering $X_{g,e+1}$ in place of $X_{g,e}$. To see this, let $Z$ be the manifold obtained by adding an additional $-1$-framed handle $h_5$ to $X$ along a meridian of $h_0$, and let $A_5 \in H_2(Z)$ be the element of square $-1$ represented by this handle. We may alternately decompose $Z$ as
\[
Z = (\mathbb{I} \conn \CPbar^2) \cup X_{g,e+1},
\]
where the first stage corresponds to adding the handle $h_4$ to $\mathbb{I} = (\#^{2g} S^1 \times S^2) \times I$, and the second stage corresponds to adding $h_0$, $h_1$, $h_2$, $h_3$, and $h_5$. Denote by $A_0', \cdots, A_4'$ the classes in this $H_2(X_{g,e+1})$ corresponding to $A_0, \dots, A_4$ in the original definition; in $H_2(Z)$, we have $A_0' = A_0+A_4$, $A_1'=A_1$, $A_2' = A_2$, $A_3'=A_3$, and $A_4'=A_5$. Observe that $Q_{2g+1,e+3}$ and $Q_{2g+2, e+1}$ sit inside both $X_{g,e}$ and $X_{g,e+1}$.

Let $\spincv_{i,j,k,l}'$ denote the spin$^c$ structure on $X_{g,e+1}$ that is defined (through its evaluations on $A_0', \dots, A_4'$) analogously to $\spincv_{i,j,k,l}$. For $a,b \in \{0,1\}$, let $\spincw_{a,b}$ be the extension of $\spincv_{g+e+a,b,0,0}$ to $Z$ satisfying
\[
\gen{c_1(\spincw_{a,b}), A_5} - 1 = 0.
\]
Then
\begin{align*}
\gen{c_1(\spincw_{a,b}|_{X_{h,e+1}}), A_0'} + e+1 &= \gen{c_1(\spincv_{g+e+a,b,0,0}), A_0 + A_4} + e+1 \\
&= (-e + 2(g+e+a)) + (1+0) + e + 1 \\
&= 2(g+e+1+a)
\intertext{and}
\gen{c_1(\spincw_{a,b}|_{X_{h,e+1}}), A_1'} &= 2b,
\end{align*}
which implies that the restriction of $\spincw_{a,b}$ to $X_{g,e+1}$ is $\spincv'_{g+e+1+a,b,0,0}$. It follows that the two naming conventions for spin$^c$ structures on $Q_{2g+1,e+3}$ and $Q_{2g+2,e+2}$ agree. \qed
\end{remark}

Additionally, we determine the restriction of $\spincv_{i,j,k,l}$ to $M_{g,e} \conn S^1 \times S^2$. Observe that the image of a generator of $H_2(V_{g,e})$ in $H_2(X)$ is equal to $A_0+2A_3$. Since
\[
\gen{c_1(\spincv_{i,j,k,l}),A_0+2A_3}+e=2(i+2k),
\]
it follows that $\spincv_{i,j,k,l}|_{M \conn S^1 \times S^2} = \hat\spincs_{i+2k}$, where $\hat\spincs_{i+2k}$ restricts to $\spincs_{i+2k}$ on $M_{g,e}$ and to the unique torsion spin$^c$ structure on $S^1 \times S^2$.

\begin{lemma}
\label{L:dbound}
For even large negative $e$, the $\dbot$ invariants of the two torsion spin$^c$ structures on $Q_{h,e}$ that do not extend over the disk bundle $P_{h,e}$ satisfy
\[
\dbot(Q_{h,e},\spinct_{a}) \le \frac{e-2}4+a, \qquad a=0,1.
\]
\end{lemma}

\begin{proof}
To simplify the notation, we write $Q$ for $Q_{2g+1,e+2}$, $Q'$ for $Q_{2g+2,e}$, and $M$ for $M_{g,e}$. We also write $\spinct_a$ for $\spinct_a^{2g+1,e+2}$ and $\spinct'_a$ for $\spinct_a^{2g+2,e}$. For $a=0,1$, let $\spincv'_a$ and $\spincv''_a$ be the restrictions of $\spincv_{g+a,1,0,0}$ to $W_2$ and $W_3$, respectively. Since $e$ is even, note that $\spincv_{g+a,1,0,0}|_Q = \spincv_{g+e+a,1,0,0}|_Q = \spinct_a$ and $\spincv_{g+a,1,0,0}|_{Q'} = \spincv_{g+e+a,1,0,0}|_{Q'} = \spinct'_a$. Consider the induced maps
\[
\HF^\circ(Q, \spinct_a) \xrightarrow{F^\circ_{W_2, \spincv'_a}} \HF^\circ(Q', \spinct_a') \xrightarrow{F^\circ_{W_3, \spincv''_a}} \HF^\circ(M \conn S^1 \times S^2, \hat\spincs_{g+a} ).
\]

The image of a generator of $H_2(W_2)$ in $H_2(X)$ is $A_2$, which has square $0$. Since the handle was attached along a nullhomologous knot, the map $F_{W_2,\spincv'}^\infty$ is an isomorphism onto the kernel of the action by the new free generator of $H_1(Q')$ (see \cite[Proposition 9.3]{oz:boundary}). In particular, $F^\infty_{W_2, \spincv'_a}$ maps the bottom tower in $\HFi(Q, \spinct_a)$ isomorphically to the bottom tower in $\HFi(Q', \spinct'_a)$. The grading shift of this map is $-1/2$, hence
\begin{equation}\label{eq:boundQ}
\dbot(Q,\spinct_a)-\frac12 \le \dbot(Q',\spinct_a').
\end{equation}

The image $G$ of a generator of $H_2(W_3)$ in $H_2(X)$ is $G=A_0+ (e/2) A_1 +2 A_3$ and has square $G^2=e$, so this cobordism is negative definite for negative $e$. Since
\[
\gen{c_1(\spincv_{i,j,k,l}),G}=2i+e(j-1)+4k,
\]
the grading shift of $F^\circ_{W_3, \spincv''_a}$ is
\[
\delta=\frac{(g+a)^2}e+\frac14.
\]
Since $F^\infty_{W_3,\spincv''_a}$ is an isomorphism (by \cite[Proposition 9.4]{oz:boundary}), we have
\[
\dbot(Q',\spinct_{g+a}') + \delta \le \dbot(M \conn S^1 \times S_2, \hat\spincs_{g+a}) = \dbot(M,\spincs_{g+a}) - \frac12,
\]
using Proposition \ref{prop:additivity}. If $e \le -2g$, the same argument as in \cite[Lemma 9.17]{oz:boundary} shows that
\[
\dbot(M,\spincs_{g+a}) = \frac{e+1}{4} + a + \frac{(g+a)^2}{e},
\]
so
\[
\dbot(Q',\spinct_a')\le \frac{e-2}{4}+a
\]
which when combined with \eqref{eq:boundQ} gives the result for $Q$ as well.
\end{proof}

We now show that the $d$-invariants of the torsion spin$^c$ structures on $Q_{h,e}$ that do not extend over $P_{h,e}$ are linear in $e$.

\begin{proposition}
\label{P:linearity}
For any $h>0$, $e\in\Z$, and $a \in \{0,1\}$, we have
\[
\dbot(Q_{h,e+1}, \spinct_a^{h,e+1}) = \dbot(Q_{h,e}, \spinct_a^{h,e}) + \frac14.
\]
\end{proposition}

\begin{proof}
We begin by considering the case where $h=2g+1$. For conciseness, we write $Q = Q_{2g+1, e+2}$ and $Q' = Q_{2g+1, e+3}$, and likewise write $\spinct_a = \spinct_a^{2g+1, e+2}$ and $\spinct'_a = \spinct_a^{2g+1, e+3}$. Note that the cobordisms $W_0 = R_{2g+1,e+2}$ and $W_4$, shown in the left-hand column of \eqref{eq:handles}, give consecutive maps in the long exact sequence for $(\infty, e, e+1)$ surgery on the attaching circle for $h_0$. Let $Z$ denote the third cobordism in the sequence, gotten by attaching a $-1$-framed $2$-handle along a meridian of the attaching circle for $h_4$. Thus, the following sequence is exact:
\begin{equation} \label{eq:exactseq}
\cdots \xrightarrow{F^+_Z} \HFp(\#^{2g+1} S^1 \times S^2) \xrightarrow{F^+_{W_0}} \HFp(Q) \xrightarrow{F^+_{W_4}}  \HFp(Q') \xrightarrow{F^+_Z} \cdots.
\end{equation}

Any spin$^c$ structure on $W_0$ that restricts to either $\spinct_0$ or $\spinct_1$ must restrict to a non-torsion spin$^c$ structure on $\#^{2g+1} S^1 \times S^2$, since otherwise we could extend it over $\natural^{2g+1} S^1 \times B_3$ to get a spin$^c$ structure on $P_{2g+1,e+2}$ extending $\spinct_i$, a contradiction. The same is true for spin$^c$ structures on $Z$.

The group $H_2(W_4)$ is generated by an element $G$ whose image in $H_2(X)$ is $A_1 + 2A_4$ and whose self-intersection is $-4$. For $a \in \Z_2$ and $m \in \Z$, let $\spincx_{a,m}$ denote the spin$^c$ structure on $W_4$ that restricts to $\spinct_a$ on $Q$ and satisfies
\[
\gen{c_1(\spincx_{a,m}), G} - 4 = 2m,
\]
It is easy to verify that the restriction of $\spincx_{a,m}$ to $Q'$ is $\spinct_{a+m}'$.

Since the summand of $\HFp(\#^{2g+1} S^1 \times S^2)$ in any non-torsion spin$^c$ structure is zero, we see from \eqref{eq:exactseq} that $F^+_{W_4}$ restricts to an isomorphism
\[
\HFp(Q, \spinct_0) \oplus \HFp(Q, \spinct_1) \to \HFp(Q', \spinct'_0) \oplus \HFp(Q', \spinct'_1).
\]
Since $W_4$ is a negative-definite $2$-handle addition, each map $F^\infty_{W_4, \spincx_{a,m}}$ is an isomorphism, so $F^+_{W_4, \spincx_{a,m}}$ takes the bottom tower in $\HFp(Q, \spinct_a)$ surjectively to the bottom tower in $\HFp(Q', \spinct'_{a+m})$. The grading shift of $F^+_{W_4, \spincx_{a,m}}$ is
\[
\frac{c_1(\spincx_{a,m})^2-2\chi(W_4) - 3\sigma(W_4)}{4} = \frac{-(m+2)^2+1}{4}.
\]
Since $F^+_{W_4, \spincx_{a,-2}}$ lowers the grading the least among all the maps $F^+_{W_4, \spincx_{a,m}}$, its restriction is injective.
Thus,
\[
\dbot(Q', \spinct'_a) = \dbot(Q', \spinct_a) + \frac14
\]
as required.

The case where $h=2g+2$ proceeds in the exact same manner, making use of the cobordisms in the middle column of \eqref{eq:handles}.
\end{proof}

Combining results of Lemma \ref{L:dbound} and Proposition \ref{P:linearity} about the $d$-invariants of the torsion spin$^c$ structures on $Q_{h,e}$ with some known embeddings into lens spaces we are able to compute the $d$-invariants in these spin$^c$ structures.

\begin{proof}[Proof of Theorem \ref{T:d}]
Recall from \cite{bredon-wood:surfaces} that $M=L(2k,1)$ contains an (essential) embedding of $F_k$. This can be easily described as follows. Let $\alpha$ and $\beta$ be sides of a square describing a genus 1 Heegaard diagram for $S^3$ and let $\gamma$ be the linear slope representing the homology class $\alpha + 2k\beta$; label the intersections between $\alpha$ and $\gamma$ by $0, 1, \ldots, 2k-1$. Then $(\alpha,\gamma)$ is a Heegaard diagram for $L(2k,1)$. An embedding of $F_k$ can be constructed by starting with the core of the 2-handle, attaching to its boundary $\gamma$ $k$ non-orientable 1-handles (that lie in the Heegaard surface and connect arcs on $\gamma$ labeled $2j$ and $2j+1$ for $j=0,\ldots,k-1$) and capping the resulting boundary off with the cocore of the 1-handle. As an embedding in $M\times I$ this surface has vanishing normal Euler number, $e=0$. The case of a Klein bottle ($h=2$) embedded in $L(4,1)$ is illustrated in Figure \ref{F:L41}.

\begin{figure}[!ht]
\centering
    \labellist
    \normalsize\hair 0mm
    \pinlabel {$0$} at 0 -7
    \pinlabel {$1$} at 21 -7
    \pinlabel {$2$} at 42 -7
    \pinlabel {$3$} at 63 -7
    \pinlabel {{\color{red}\footnotesize$\alpha$}} at 15 12
    \pinlabel {\footnotesize$\beta$} at -8 45
    \pinlabel {\footnotesize$\gamma$} at 11 65
    \endlabellist
\includegraphics[scale=1.2]{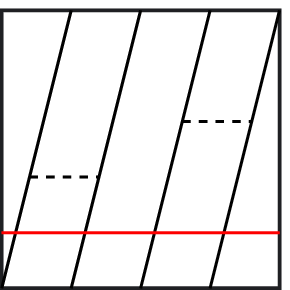}
\caption{Lens space $L(4,1)$.}
  \label{F:L41}
\end{figure}

We will apply Theorem \ref{T:twist-bound} to this embedding. Recall from \cite[Proposition 4.8]{oz:boundary} that the spin$^c$ structures on $L(2k,1)$ can be labeled $\spincs_0, \dots, \spincs_{2k-1}$ so that the $d$-invariants are
\[
d(M,\spincs_s)=\frac14-\frac{(s-k)^2}{2k},
\]
and the relevant differences of these invariants are
\[
d(M,\spincs_{s+k})-d(M,\spincs_s) = \frac{k}{2} - s, \qquad s=0,\ldots,k-1.
\]
The maximal difference is attained for the spin$^c$ structure $\spincs_0$, and the right-hand inequality in \eqref{dbounds} then gives $\dbot(Q_{k,0},\spinct_{\spincs_0}) \ge 1/2$. Combining the bound of Lemma \ref{L:dbound} with the linearity result of Proposition \ref{P:linearity} yields $\dbot(Q_{k,0},\spinct_0^{k,0}) \le -1/2$ and $\dbot(Q_{k,0},\spinct_1^{k,0}) \le 1/2$. This implies that $\spinct_{\spincs_0} = \spinct_1^{k,0}$ and
$\dbot(Q_{k,0},\spinct_1^{k,0}) = 1/2$. Another application of linearity gives $\dbot(Q_{k,e},\spinct_1^{k,e}) = (e+2)/4$.

By Proposition \ref{twist}, $\spinct_{\spincs_1} \ne \spinct_{\spincs_0}$, so $\spinct_{\spincs_1} = \spinct^{k,0}_0$.
Thus, \eqref{dbounds} gives $\dbot(Q_{k,0},\spinct_{0})\ge -1/2$, which in combination with the upper bound and linearity gives $\dbot(Q_{k,e},\allowbreak \spinct_{0}^{k,e})=(e-2)/4$.

To determine $\dtop(Q_{k,e},\spinct_a)$, first note that there is an orientation-reversing diffeomorphism between $Q_{h,e}$ and $Q_{h,-e}$. By Proposition \ref{prop:duality}, for $a = 0,1$, we either have
\begin{equation} \label{eq:a}
\dtop(Q_{h,e}, \spinct_a^{h,e}) = -\dbot(Q_{h,-e}, \spinct_a^{h,-e}) = \frac{e+2}{4} - a
\end{equation}
or
\begin{equation} \label{eq:1-a}
\dtop(Q_{h,e}, \spinct_a^{h,e}) = -\dbot(Q_{h,-e}, \spinct_{1-a}^{h,-e}) = \frac{e+2}{4} + a - 1,
\end{equation}
depending on how the diffeomorphism acts on the set of spin$^c$ structures. Also, by Lemma \ref{lemma:topbot}, we have
\[
\dtop(Q_{h,e}, \spinct_a^{h,e}) \equiv \dbot(Q_{h,e}, \spinct_a^{h,e}) - h+1 \pmod 2.
\]
Thus, \eqref{eq:a} holds when $h$ is even, and \eqref{eq:1-a} holds when $h$ is odd, as required.

The final statement in the theorem follows immediately from Corollary \ref{C:QHS1xB3}.
\end{proof}

\section{Proof of Theorem~\ref{T:Mbound}}\label{S:Hcob}
Theorem \ref{T:twist-bound}, which provides genus bounds for surfaces in a homology cobordism $(W,M_0,M_1)$ with the $M_i$ rational homology spheres, is most effective when we apply it to the spin$^c$ structure that maximizes the value of the differences of $d$-invariants appearing in the statement of the theorem. This motivates the following definition.

\begin{definition}
\label{D:delta}
Let $M$ be a rational homology sphere and $\phi\in H^2(M;\Z)$ a nontrivial class of order $2$. Define
$$\Delta=\Delta(M,\phi)=\max\{d(M,\spincs+\phi)-d(M,\spincs)\mid \spincs\in \spinc (M)\}.$$
\end{definition}
By the homology cobordism invariance of $d$-invariants, $\Delta$ is also a homology cobordism invariant.  In particular, if $W$ is a homology cobordism between $M_0$ and $M_1$, the spin$^c$ structures on the ends are canonically identified, and $\Delta(M_0,\phi) = \Delta(M_1,\phi)$.

Recall from Proposition \ref{twist} that if $F \subset W$ is an essential embedding in a homology cobordism, then a curve in $F$ representing the torsion generator $c \in H_1(F;\Z)$ represents a nontrivial class $[c] \in H_1(W;\Z)$. It is the Poincar\'e dual of this class, $\phi=\PD([c])$, that appears in the bound of Theorem \ref{T:twist-bound}.  For the reader's convenience, we restate Theorem~\ref{T:Mbound} from the introduction, with a slight refinement.
\begin{theorem}
Suppose that $(W,M_0,M_1)$ is a homology cobordism between rational homology spheres, and that $F_h\subset W$ is essential and has normal Euler number $e$. Let $\Delta=\Delta(M_0,\phi)$. Then
\[
h \ge 2\Delta;
\]
furthermore, if $h$ is even and the restriction of $\phi$ to $F_h$ is nontrivial, or if $h$ is odd and the restriction of $\phi$ to $F_h$ is trivial, then
\[
h \ge 2\Delta+1.
\]
Additionally,
\[
|e| \le 2h - 4\Delta\qquad \text{and}\qquad e \equiv 2h - 4\Delta \pmod4 .
\]
\end{theorem}

\begin{proof}
Let $\spincs\in\spinc (M_0)$ be a spin$^c$ structure on $M_0$ (and hence on $W$ and $M_1$) such that $d(M_1,\spincs+\phi) - d(M_0,\spincs) = \Delta$.
Let $\spinct = \spinct_{\spincs}$ and $\spinct' = \spinct_{\spincs+ \phi}$. Note that $\spinct = \spinct'$ if and only if the restriction of $\phi$ to $Q$ is trivial. From \eqref{dbounds} we have
\begin{equation}\label{E:delta1}
\dtop(Q,\spinct)-\frac{h-1}2 \le \Delta \le \dbot(Q,\spinct)+\frac{h-1}2.
\end{equation}
Note that since $\phi$ is of order 2, replacing $\spincs$ by $\spincs + \phi$ in \eqref{dbounds} changes the sign of the middle term, thus giving
\begin{equation}\label{E:delta2}
\dtop(Q,\spinct')-\frac{h-1}2 \le -\Delta \le \dbot(Q,\spinct')+\frac{h-1}2.
\end{equation}
Subtracting the two inequalities yields
\[
2\Delta \le \dbot(\spinct)-\dtop(\spinct')+h-1.
\]
By Theorem \ref{T:d}, we have $\dbot(\spinct) - \dtop(\spinct') \le 1$, implying that $h \ge 2\Delta$. Furthermore, if $h$ is even and $\spinct\ne\spinct'$, or if $h$ is odd and $\spinct=\spinct'$, then $\dbot(\spinct) = \dtop(\spinct')$, implying that $2 \Delta \ge h+1$.

To get the bounds on the normal Euler number, use \eqref{E:delta1} combined with the facts that $\dbot(\spinct) \le (e+2)/4$ and $\dtop(\spinct) \ge (e-2)/4$. Finally, reducing the congruence condition in Theorem \ref{T:twist-bound} modulo 1 yields the congruence condition on $e$.
\end{proof}

More can be said in the case of embeddings of the projective plane in a homology cobordism $W$.

\begin{corollary}\label{C:rp2}
If $\RP^2 \subset W$ is essential, then the Euler number of the embedding must be $0$ and $\phi$ restricts nontrivially to $\RP^2$. Let $K$ be the kernel of the restriction homomorphism $H^2(W; \Z) \to H^2(\RP^2;\Z)$. Then there exists $\spincs_0 \in \Spin^c(M_0)$ such that
\[
d(M_0,\spincs+\phi)-d(M_0,\spincs)=\frac12
\]
for each $\spincs \in \spincs_0+K$. (That is, the $d$-invariants of $M_0$ are the same as those of a manifold with an $\RP^3$ summand.)
\end{corollary}

\begin{proof}
In the special case where $h=1$, \eqref{dbounds} becomes
\[
\Delta(\spincs)=d(M,\spincs+\phi)-d(M,\spincs)=d(Q,\spinct_\spincs)
\]
for any $\spincs\in\Spin^c(M)$. Using $\spincs+\phi$ in place of $\spincs$ gives
\[
-\Delta(\spincs)=d(Q,\spinct_{\spincs+\phi}),
\]
and the two combine to
\[
d(Q,\spinct_{\spincs+\phi})=-d(Q,\spinct_\spincs)
\]
for all $\spincs\in\Spin^c(M)$. If the restriction of $\phi$ to the surface were trivial, this would imply
\[
\frac{e-2}4=-\frac{e-2}4\qquad\text{and}\qquad \frac{e+2}4=-\frac{e+2}4,
\]
a contradiction. Hence $\phi$ restricts nontrivially to the surface and
\[
\frac{e-2}4=-\frac{e+2}4\]
holds, implying $e=0$.

Now choosing $\spincs_0$ such that $d(Q,\spinct_{\spincs_0})=1/2$ gives the result.
\end{proof}


%
%
%
%

\section{Connections with rational genus} \label{S:rational}

For a knot $K$ representing a torsion homology class of order $p \ge 1$ in a $3$-manifold $M$, the \emph{rational genus} of $K$ is defined as
\[
g_r(K) = \min_{G}\frac{-\chi(G)}{2p},
\]
where the minimum is taken over all properly embedded, connected, oriented surfaces $G$ in $M \smallsetminus N$ such that $\partial G$ is homologous to $p$ times $K$ in the interior $N$ of a small tubular neighborhood of $K$. Such surfaces $G$ are called \emph{rational Seifert surfaces}.\footnote{Note that this definition differs slightly from the standard one given by Turaev \cite{turaev:homology} and Calegari--Gordon \cite{calegari-gordon:rational-genus} in that we do not exclude disks. Specifically, if there is a disk $G \subset M \smallsetminus N$ whose boundary winds along $K$ $p$ times, then $g_r(K) = -\frac{1}{2p}$ according to our definition, whereas ordinarily the rational genus of such a knot is defined to be zero. Of course, in this case, $K$ must be contained in an $L(p,q)$ summand of $M$.} Following Ni and Wu \cite{ni-wu:rational-genus}, for a torsion class $x \in H_1(M;\Z)$, define
\[
\Theta(x) = 2 \min_{K \subset M \mid [K] = x} g_r(K).
\]
This notion, for order-2 homology classes $x$, is closely connected with the embedding problem studied in the present paper, as we now explain.

As in the introduction, if $M$ is a rational homology sphere and $\phi \in H^2(M;\Z)$, define
\[
\Delta(M,\phi) = \max\{ d(M, \spincs + \phi) - d(M, \spincs) \mid \spincs \in \Spin^c(M) \}.
\]
Also, recall that $M$ is an \emph{L-space} if $\HFhat(M, \spincs) \cong \Z$ for each spin$^c$ structure $\spincs$ on $M$. A knot $K$ in an L-space $M$ (which need not be nulhomologous) is called \emph{Floer-simple} if $\HFKhat(M, K, \spincs) \cong \HFhat(M, \spincs) \cong \Z$ for each $\spincs$. For example, lens spaces are L-spaces, and every homology class in a lens space can be represented by a Floer-simple knot \cite{hedden:berge, rasmussen:lens}.

Ni and Wu \cite[Theorem 1.1]{ni-wu:rational-genus} proved that for any rational homology sphere $M$ and any $x \in H_1(M;\Z)$
\begin{equation} \label{E:Theta}
\Delta(M,\PD(x)) \le 1 + \Theta(x).
\end{equation}
Furthermore, if $M$ is an L-space and $K$ is a Floer-simple knot, then $K$ minimizes genus in its homology class, and the bound \eqref{E:Theta} is an equality \cite[Theorem 1.2 and Proposition 5.1]{ni-wu:rational-genus}.

For a rational homology sphere $M$, let $\beta\co H_2(M;Z_2) \to H_1(M;\Z)$ denote the connecting homomorphism in the Bockstein sequence associated to $0 \to \Z \to \Z \to \Z_2 \to 0$. Note that $\beta$ is an injection whose image consists of all $2$-torsion elements in $H_1(M)$.

\begin{lemma} \label{L:rationalgenus}
Let $M$ be a rational homology sphere, and let $a \in H_2(M;\Z_2)$ be a nonzero homology class. Then the minimum genus of any connected, non-orientable surface representing $a$ is equal to $2+2\Theta(\beta(a))$.
\end{lemma}

\begin{proof}
Suppose that $F = F_h \subset M$ is a minimal-genus embedded surface representing $a$. As discussed in Section~\ref{S:prelim}, $F$ represents a nontrivial class in $H_2(M;\Z_2)$, and the torsion class in $H_1(F;\Z)$ maps to $\beta(a)$. Let $C \subset F$ be an embedded curve representing this homology class, which we may view as a knot in $M$, and let $N$ be a regular neighborhood of $C$. Removing the neighborhood $N$ yields a properly embedded, orientable surface $F' \subset M \smallsetminus N$ with the Euler characteristic $\chi(F') = \chi(F) = 2-h$, and hence
\[
1 + \Theta(\beta(a)) \le 1 + 2 g_r(C) = 1 + \frac{h-2}{2} = \frac{h}{2}.
\]

Conversely, suppose $K$ is a genus-minimizing knot representing the class $\beta(a)$, and $F'$ is a rational Seifert surface so that $-2\chi(F') = \Theta(\beta(a))$. The boundary of $F'$ is a $(2,m)$ cable of $K$, which we can fill in with either an annulus (if $m$ is even) or a M\"obius band (if $m$ is odd) to obtain a closed, non-orientable surface $F$ representing $a$, with genus
\[
h = 2 - \chi(F) = 2-\chi(F') = 2+2\Theta(\beta(a)). \qedhere
\]
\end{proof}

The work of Ni and Wu, combined with Lemma \ref{L:rationalgenus}, implies:

\begin{theorem} \label{T:ni-wu}
Let $M$ be a rational homology sphere, and let $a \in H_2(M;\Z_2)$. If $F_h$ embeds into $M$ representing the class $a$, then
\begin{equation} \label{E:ni-wu}
h \ge 2\Delta(M, \PD(\beta(a))).
\end{equation}
Furthermore, if $M$ is an $L$ space and $\beta(a)$ is represented by a Floer-simple knot, then there exists an embedding of $F_h$ representing $a$ with
\[
h = 2\Delta(M, \PD(\beta(a))). \qedhere
\]
\end{theorem}

\begin{corollary}\label{C:simple}
Let $M$ be an $L$-space, and suppose $a\in H_2(M;\Z_2)$ is a class such that $\beta(a)$ is represented by a Floer-simple knot. Let $\Delta=\Delta(M,\PD(\beta(a)))$. For any homology cobordism $W$ with one boundary $M$ (e.g. $W = M \times I$), there is an embedding of $F_h$ in $W$ with normal Euler number $e$ representing $a$ if and only if
\begin{equation} \label{E:hbounds}
h \ge 2\Delta, \quad \abs{e} \le 2h - 4\Delta \quad \text{and} \quad e \equiv 2h - 4\Delta \pmod4.
\end{equation}
\end{corollary}

\begin{proof}
The ``only if'' direction follows immediately from Theorem \ref{T:Mbound}. For the ``if'' direction, the second half of Theorem \ref{T:ni-wu} implies that there exists an embedding of $F_{2\Delta}$ in $W$ with Euler number $0$. For any $(h,e)$ satisfying \eqref{E:hbounds}, we can construct an embedding of $F_h$ with Euler number $e$ as follows. Let $\ell=h-2\Delta$. The congruence conditions for genus $2\Delta$ with Euler class $0$ and genus $h$ with Euler class $e$ imply that $e \equiv 2\ell \pmod 4$. Then an embedding of $F_h$ into $W$ can be constructed from that of $F_{2\Delta}$ by taking the pairwise connected sum with $(2\ell+e)/4$ copies of an embedding of $\RP^2$ in $S^4$ with Euler number $+2$ and $(2\ell-e)/4$ copies of an embedding of $\RP^2$ in $S^4$ with Euler number $-2$.
\end{proof}

We now consider several classes of manifolds to which the results of this section may be applied.

\subsection{Lens spaces}\label{S:lens}

For any $1 \le q \le k$ with $q$ relatively prime to $2k$, Bredon and Wood \cite{bredon-wood:surfaces} showed using elementary geometric techniques that the minimal genus of a non-orientable surface embedded in the lens space $L(2k,q)$ is equal to $N(2k,q)$, where the function $N$ is defined recursively by:
\begin{itemize}
\item $N(2,1) = 1$
\item $N(2k,q) = N(2(k-q), q') +1$, where $q' \equiv \pm q \pmod{ 2(k-q)}$ and $1\leq q' \leq k-q$.
\end{itemize}
Because lens spaces contain Floer-simple knots in each homology class, it follows from Theorem \ref{T:ni-wu} that $N(2k,q) = 2\Delta(L(2k,q),\phi)$, where
$$
\phi = k \in H^2(L(2k,q);\Z) \cong \Z_{2k}
$$
is the unique element of order $2$.

It is worth noting that the differences of $d$-invariants that go into the definition of $\Delta$ can be computed
quite explicitly for lens spaces using a formula of Lee and Lipshitz \cite{lee-lipshitz:q-gradings} for the relative grading between two generators of the Heegaard Floer complex of a Heegaard diagram. Specifically, Ozsv\'ath and Szab\'o \cite[Section 4.1]{oz:boundary} give a particular labeling of the spin$^c$ structures on $L(p,q)$ by $\spincs_0, \dots, \spincs_{p-1}$, where under a certain identification of $H_1(L(p,q);\Z)$ with $\Z_p$, we have $\s_i + j = \s_{i+j}$. The Heegaard Floer complex associated to the standard Heegaard diagram for $-L(p,q)$ has exactly $p$ generators $x_0, \dots, x_{p-1}$, where $x_i$ represents $\spincs_i$. Thus, $d(-L(p,q), \spincs_i) = \grbar(x_i)$. By~\cite[Corollary 5.2]{lee-lipshitz:q-gradings}, we have
\begin{equation}\label{E:gr-shift}
d(-L(p,q), \spincs_{i+q}) - d(-L(p,q), \spincs_i) = \grbar(x_{i+q}) - \grbar(x_i) = \frac1p(p-1-2i).
\end{equation}
For $j \in \Z$, let $[j]$ be the integer congruent to $j$ modulo $p$ satisfying $0 \leq [j] \leq p-1$. Now let $p = 2k$, and note that $kq \equiv k \pmod{2k}$.  Applying~\eqref{E:gr-shift} $k$ times gives
\begin{equation}\label{E:d-shift}
d(-L(2k,q), \spincs_{i+k}) - d(-L(2k,q), \spincs_i) = \frac{2k-1}{2} - \frac1{k}  \sum_{j=0}^{k-1} [i + qj].
\end{equation}
Denote the function on the right-hand side of \eqref{E:d-shift} by $g(2k,q,i)$, or by $g(i)$ if $k$ and $q$ are understood from context. Setting
\[
G(2k,q) = \max\{g(2k,q,i)\mid i\in \Z_{2k}\},
\]
it follows that
\[
\Delta(L(2k,q),\phi) = \Delta(-L(2k,q),\phi) = G(2k,q).
\]
In the Appendix, we present a number-theoretic proof by Ira Gessel that $2G(2k,q)$ satisfies the same recursion relation as $N(2k,q)$, and thus that the two quantities are equal for all $(k,q)$. Combined with Theorem \ref{T:Mbound}, this provides a new proof (independent of \cite{bredon-wood:surfaces} and \cite{ni-wu:rational-genus}) that $N(2k,q)$ gives a lower bound on the genera of non-orientable surfaces in $L(2k,q)$.

\subsection{Strong L-spaces}

A Heegaard diagram $\mathcal{H} = (\Sigma, \bm\alpha, \bm\beta)$ for a rational homology sphere $M$ is called \emph{strong} if the rank of the associated Heegaard Floer complex $\CFhat(\mathcal{H})$ is equal to $\abs{H_1(M;\Z)}$; we call $M$ a \emph{strong L-space} if it admits a strong Heegaard diagram \cite{levine-lewallen:strong}. Suppose that $M$ admits a strong diagram of genus $2$. Forthcoming work of Josh Greene and the first author will show that every class in $H_1(M;\Z)$ can be represented by a Floer-simple knot. (Furthermore, such $M$ must be a graph manifold with two Seifert fibered pieces, each of which fibers over a disk with two exceptional fibers.) By Corollary \ref{C:simple}, the minimal genus problem for non-orientable surfaces in any homology cobordism from $M$ to itself is the same as the minimal genus problem in $M$. Additionally, just as in the previous section, $\Delta(M,\phi)$ can be easily determined from the strong Heegaard diagram using the Lee--Lipshitz formula. The authors do not know whether these results can be extended to arbitrary strong L-spaces.

\subsection{Homology classes that do not contain Floer-simple knots}

In the opposite direction, Theorem \ref{T:ni-wu} can be used to prove that not every homology class in an L-space can be represented by a Floer-simple knot.

\begin{proposition}\label{P:non-simple}
There is a rational homology sphere $Y$ with $H_1(Y) = \Z_6$ that is an L-space, but for which there is no Floer-simple knot in the non-trivial order $2$ homology class.
\end{proposition}

\begin{proof}
Let $Y$ be the Seifert-fibered space $M(-1; (3,2),(4,1),(6,1))$ (we follow notation in \cite{orlik}), whose homology is easily computed to be $\Z_6$.   This manifold can be verified to be an $L$-space using the criterion of Lisca and Stipsicz~\cite{lisca-stipsicz:tight-III}.  (This was independently confirmed by a computer calculation of Jonathan Hanselman based on bordered Floer homology.)  Computing the $d$-invariants of $Y$ via the \oz\ algorithm~\cite{oz:plumbed} yields that $\Delta(Y) = 1/2$.  If $Y$ contained a Floer-simple knot, then Theorem \ref{T:ni-wu} would give rise to an embedded $\RP^2$ carrying the non-trivial class in $H_2(Y;\Z_2)$.  But the existence of an embedded $\RP^2$ would imply that $Y$ is a connected sum $\RP^3 \conn Y'$ where $H_1(Y') \cong \Z_3$. But $Y$, being a Seifert-fibered space whose base is hyperbolic~\cite[Section 3]{scott:geometries}, is irreducible, so this cannot happen.
\end{proof}

\subsection{Mappings versus embeddings}\label{S:maps}
One striking consequence of Gabai's work relating foliations and the Thurston norm~\cite{gabai:foliations} is that the Thurston norm of an integral homology class (roughly, the minimal genus of an embedded representative) is the same as the singular Thurston norm (the minimal genus of any surface that maps to $M$ in the given homology class)~\cite[Corollary 6.18]{gabai:foliations}. Since the projection $M \times I \to M$ induces an isomorphism on homology, this means that one cannot lower the genus by embedding in $M \times I$ instead of embedding in $M$. The results of this paper give evidence for a non-orientable analogue of Gabai's result, that the minimal genus of a non-orientable representative of a homology class in $M \times I$ is the same as the minimal genus in $M$.

However, the following proposition implies that such a result cannot be proved by looking at non-orientable surfaces mapping to $M$, as in the orientable case:
\begin{proposition}\label{mapping}
For any $k$ and $q$, there is a map from $\RP^2$ to $L(2k,q)$ inducing the non-trivial map in $\Z_2$ homology.
\end{proposition}

\begin{proof}
View $L(2k,q)$ as usual as the union $S^1\times D^2 \cup_\phi D^2 \times S^1$ where $\phi(\partial D^2  \times p)$ is a $(2k,q)$ curve $K$ on $\partial(S^1\times D^2)$. Assume $K$ is the standard $(2k,q)$ torus knot, given by parametrization $S^1 \to S^1\times\partial D^2$, $\zeta \mapsto (\zeta^{2k},\zeta^q)$. Then an immersed M\"obius band $\beta$ with boundary $K$ in $S^1\times D^2$ is constructed by connecting points $(z,w)$ and $(z,-w)$ on $K$ by a line segment in $\{z\} \times D^2$. By adding the 2-cell of $L(2k,q)$ to $\beta$ we obtain an immersed $\RP^2$ that carries the $\Z_2$ homology of the lens space.
\end{proof}
The relation between the singular Thurston norm and the fundamental group of $M$ is also discussed in~\cite[Section 4] {calegari-gordon:rational-genus}.

\section{Genus bounds from the \texorpdfstring{$\rho$}{rho}-invariants}\label{S:sign}
The twisting of spin$^c$ structures described in Proposition~\ref{twist} gives rise to embedding obstructions stated in terms of classical Atiyah-Singer invariants arising from the G-signature theorem~\cite{atiyah-singer:III}.   The idea is similar to the classic paper of Massey~\cite{massey:whitney}; one considers a branched cover and compares the result of the G-signature theorem with a Smith-theory estimate of the equivariant signature.   Both the genus and the Euler class appear in the bounds, so Theorem~\ref{T:g-sign} below can be read as providing a restriction on the genus for fixed Euler class (or vice versa).  These obstructions differ from the bounds in Theorem~\ref{T:twist-bound} and are generally not as strong as those arising from considerations of $d$-invariants.  For instance, as we will see in Example~\ref{E:drho}, we cannot recover the results on surfaces in lens spaces via the signature obstructions.  On the other hand, we will also give an example where Theorem~\ref{T:g-sign} gives a stronger embedding restriction than the $d$-invariant bound.  Because the G-signature theorem holds in the locally-flat setting, Theorem~\ref{T:g-sign} applies to topologically locally-flat embeddings.  Hence, for this section, the homology cobordism $W$ is allowed to be merely a topological manifold, and the surface $F$ is allowed to be merely locally flat rather than requiring it to be smooth.

To construct the branched cover, we begin by reinterpreting Proposition~\ref{twist} in terms of $U(1)$ representations.
Let $\gamma \in H^2(V,M_0;\Z)$ be the class constructed in the proof of Proposition~\ref{twist}.  Since $\gamma$ is a $2$-torsion class, there is a class $\tau \in
H^1(V,M_0;\Z_2)$, easily seen to be unique, with $\beta(\tau) = \gamma$.   Then $\tau$ may be viewed as a $U(1)$ representation taking values in the $\Z_2 = \{\pm 1\}$ subgroup of $U(1)$; by naturality of the Bockstein, the restriction of $\tau$ to $M_1$ and to the homology class of the fiber in $Q$ is non-trivial.

Now for any $U(1)$ representation $\alpha\co \pi_1(M) \to U(1)$, we can obtain a new representation
\[
\altau\co \pi_1(V) \to U(1)
\]
by extending $\alpha$ to $W$, restricting to $V$, and then defining
\[
\altau(g) = \alpha(g) \cdot \tau(g).
\]
By construction, $\left.\altau_{}\right|_{M_0} =\left.\alpha_{}\right|_{M_0}$, but the restrictions to $M_1$ and $Q$ of $\altau$ are different from the restrictions of the extension of $\alpha$. For instance, if the image of $\left.\alpha_{}\right|_{M_0}$ has odd order, say $m$, then the image of $\left.\altau_{}\right|_{M_1}$ has order $2m$.
We will refer to a representation $\pi_1(Q) \to U(1)$ whose value on the homology class $f$ of the fiber of $Q$ is $-1$ as a {\em twisted} representation; by construction $\left.\altau_{}\right|_Q$ is twisted.

This observation gives rise to obstructions expressed in terms of an invariant due to Atiyah and Singer~\cite{atiyah-singer:III}. There are many notations for this invariant; we use the version in~\cite{aps:II}.
\begin{definition}\label{rho}
Let $M$ be an oriented $3$-manifold, and $\alpha\co H_1(M) \to U(1) $ be a representation with image in the cyclic group $\Z_m \subset U(1)$ generated by $\omega = \exp(2 \pi i/m)$.  For some $n \geq 1$, there is a $4$-manifold $X$ with $\partial X = n \cdot M$ and a representation $\alpha\co H_1(X) \to U(1)$ extending $\alpha$.  The representation $\alpha$ defines a local coefficient system $\C_\alpha$ on $X$, and we consider the signature $\sign_\alpha(X)$ defined by the intersection form on $H_2(X; \C_\alpha)$.  Then
\[
\rho_\alpha(M) = \frac1n\left(\sign(X) - \sign_\alpha(X)\right).
\]
\end{definition}
The Atiyah-Patodi-Singer index theorem~\cite{aps:II} or the G-signature theorem can be used to show that $\rho_\alpha(M)$ is independent of $n$ and the choice of $X$.  The representation $\alpha$ determines covering spaces $\widetilde{M} \to M$ and $\xtilde \to X$, with a choice $T$ of generator of the covering transformations.  With respect to that choice, the signature $\sign_\alpha(X)$ is the same as the signature of the intersection form on the $\omega$--eigenspace of $T_*$ acting on $H_2(\xtilde;\C)$.

Instead of extending $n \cdot \alpha$ over $X$, we can extend $\widetilde{M} \to M$ to a branched covering $\ztilde \to Z$, in which $T$ acts as a covering transformation of $\ztilde$ with fixed point set a locally flat surface $\widetilde{C}\subset \ztilde$, and use the G-signature theorem to compute $\rho_\alpha(M)$.  In the special case that $m=2$, this gives (compare~\cite{casson-gordon:stanford} for the general case)
\begin{equation}\label{cg-branch}
\rho_\alpha(M) = \sign(Z) - \sign_\alpha(Z)  - \ctilde^2 = \sign(Z) - \sign_\alpha(Z) -\frac12\,C^2 .
\end{equation}
In this special case, it is not necessary that $C$ be orientable, as long as the self-intersection is interpreted as in Section~\ref{S:prelim}.

For the proof of Theorem~\ref{T:g-sign}, we will need the $\rho$-invariant for a twisted representation $\alpha:\pi_1(\qhe) \to U(1)$ in the case that $e$ is even.
\begin{proposition}\label{rhoQ}
Let $e$ be even, and let $\alpha\co H_1(Q_{h,e}) \to U(1)$ be a twisted representation.  With $\qhe$ oriented as the boundary of the disk bundle $P_{h,e}$, we have
\begin{equation}\label{E:rhoQ}
\rho_\alpha(Q_{h,e}) =  -\frac{e}{2}.
\end{equation}
\end{proposition}

Our proof requires a preliminary lemma concerning the $U(1)$ representation variety of $\pi_1(\qhe)$.  Let $R^t$ denote the subset of twisted representations.
\begin{lemma}\label{L:coho}
Let $e$ be even.  For any $\alpha \in R^t(\qhe)$, the twisted cohomology group $H^1(Q_{h,e};\C_\alpha)$ vanishes.
\end{lemma}
\begin{proof}
We give the proof---a direct calculation---when $h = 2g+1$ is odd; the case when $h$ is even is only slightly different and we address it at the end. We start with a standard presentation of the fundamental group
\begin{multline*}
\pi_1(Q_{2g+1,e}) =
 \langle c,f, a_1,b_1,\ldots,a_g,b_g\ |\\
  \prod_{i=1}^g [a_i,b_i] c^2 f^e, c f c^{-1} f, [a_i,f],\  [b_i,f] \ \text{for}\ i=1,\ldots,g \rangle.
\end{multline*}
Abelianizing this presentation gives an alternate proof of Lemma~\ref{L:Q}. The coboundary operator
$$
\delta_2\co C^1(\pi_1(Q_{2g+1,e});\C_\alpha) \to C^2(\pi_1(Q_{2g+1,e});\C_\alpha)
$$
may be obtained in two steps:

\begin{enumerate}
\item Take the Fox derivatives~\cite{fox:calculus-I,brown:cohomology} of the relations with respect to the generators in the above presentation.
\item Replace each generator by its image under $\alpha$.
\end{enumerate}
The result is displayed below for $g=2$; the general case is similar. To simplify the notation, we have written $x$ for $\alpha(x)$, and substituted $-1$ for $\alpha(f)$.  The columns correspond to the generators and the rows to the relators, in the order written in the presentation.
\begin{equation*}
\delta_2 =
\begin{pmatrix}
1+c & 0 & 1-b_1 & a_1-1& 1-b_2 & a_2-1\\
2 & c-1      & 0 & 0 & 0 & 0\\
0 & a_1-1 & 2 & 0 & 0 & 0\\
0 & b_1-1 & 0 & 2 & 0 & 0\\
0 & a_2-1 & 0 & 0 & 2 &  0\\
0 & b_2-1 & 0 & 0 & 0 & 2\\
\end{pmatrix}
\end{equation*}
Since $e$ is even, the generator $c$ is of order $2$ in $H_1(\qhe)$, and hence $\alpha(c) = \pm 1$.  It is easy to see that the null space of $\delta_2$ has dimension $1$, and in fact coincides with the image of $\delta_1\co C^0 \to C^1$, which is given by the transpose of $(c-1, f-1, a_1 - 1, b_1 - 1, a_2 - 1, b_2 - 1)$. Hence the cohomology vanishes.

If $h = 2g$ is even, then we do the same calculation, based on the presentation
\begin{multline*}
\pi_1(Q_{2g,e}) =
 \langle f, a_1,b_1,\ldots,a_g,b_g\ |\\
  a_1b_1a_1^{-1}b_1\prod_{i=2}^g [a_i,b_i] ,\ a_1 f a_1^{-1} f, [b_1,f],\ [a_i,f],\  [b_i,f] \ \text{for}\ i=2,\ldots,g \rangle.
\end{multline*}
\end{proof}

\begin{proof}[Proof of Proposition~\ref{rhoQ}]
Again, we treat the case $h$ odd in detail (with notation as in Lemma~\ref{L:coho}) and add a brief comment on $h$ even at the end.
We claim first that the representation variety $R^t(Q_{h,e})$ has two path components, $R^t_\pm$ determined by the sign of $\alpha(c)$.  Certainly there is no path joining an element of $R^t_+$ to one in $R^t_{-}$ because $c$ having order $2$ implies that $\alpha(c) = \pm 1$ must be constant along any path of representations.
To see that both $R^t_\pm$ are connected, note that $U(1)$ is connected, and so  connecting $\alpha(a_i)$ and $\alpha(b_i)$ to $1 \in U(1)$ gives a path from $\alpha$ to the representation $\alpha_\pm$ defined by $\alpha_\pm(c) = \pm 1$, $\alpha_\pm(f) = -1$ and $\alpha_\pm(a_i) = \alpha_\pm(b_i) = 1$.

A path $\alpha_t$ between two representations $\alpha_0$ and $\alpha_1$ defines a family of self-adjoint operators corresponding to the signature operator, and the difference in $\rho$-invariants, $\rho_{\alpha_1}(Q) - \rho_{\alpha_0}(Q)$ is given by the spectral flow~\cite{aps:II} of this family; compare~\cite[Theorem 7.1] {kkr} for a careful discussion.  However, Lemma~\ref{L:coho} implies that $H^1(Q_{2g+1,e};\C_{\alpha_t}) = 0$ for all $t$, so there is no spectral flow for the signature operator along that path, and it suffices to calculate $\rho_{\alpha_\pm}(Q)$.

The $2$-fold covering of $Q_{2g+1,e}$ corresponding to $\alpha_{+}$ is a fiber-preserving map $Q_{2g+1,e/2} \to Q_{2g+1,e}$, which extends to a branched covering $P_{2g+1,e/2} \to P_{2g+1,e}$.  (Recall that $e$ is even, so this makes sense.)  Since $H_2(P_{2g+1,e}) = 0= H_2(P_{2g+1,e/2})$, the signature terms in Equation~\eqref{cg-branch} vanish, so $\rho_{\alpha_{+}}(Q_{2g+1,e}) = -e/2$.

The $2$-fold covering of $Q_{2g+1,e}$ corresponding to $\alpha_{-}$ also extends to a branched covering, as follows. Let $\qtilde$ be the $\Z_2 \oplus \Z_2$ covering of $Q_{2g+1,e}$ corresponding to the kernel of the surjection $\Phi= (\Phi_1,\Phi_2)\co H_1(Q_{2g+1,e}) \to \Z_2 \oplus \Z_2$ taking $c$ to $(1,0)$ and $f$ to $(0,1)$, and vanishing on the other generators in the above presentation.  Write $T_1$ and $T_2$ for the generators of the covering transformations corresponding to $c$ and $f$, respectively.  Then $\qtilde$ can be built in two steps as the boundary of a disk bundle: first take the branched cover $P_{2g+1,e/2}$ corresponding to $\Phi_2 = \alpha_{+}$.  Note that the composition $H_1(Q_{2g+1,e/2}) \to  H_1(Q_{2g+1,e}) \overset{\Phi_1}{\rightarrow} \Z_2$ extends over  $P_{2g+1,e/2}$, giving rise to an unbranched $2$-fold cover $\widetilde{P}$.  This is the disk bundle of Euler class $e$ over the orientable double cover of $F_{2g+1}$.  Now if we take $Q'= \qtilde/(T_1\circ T_2)$, the double cover $Q' \to Q_{2g+1,e}$ corresponds to $\Phi_1 \Phi_2 = \alpha_{-}$ and has covering transformation induced by $T_2$.

Now $T_1\circ T_2$ extends to a free involution on $\widetilde{P}$ with quotient the Euler class $e/2$ bundle over $F$, and $T_2$ gives an involution on this quotient with fixed point set the $0$-section.  So as above, the $2$-fold covering of $Q_{2g+1,e}$ corresponding to $\alpha_{-}$ extends to a branched covering, and  Equation~\eqref{cg-branch} implies that $\rho_{\alpha_{-}}(Q_{2g+1,e}) = -e/2$.

If $h$ is even, there are again two components $R^t_\pm$, determined by the sign of $\alpha(b_1)$.  Then one has to compute two representative $\rho$-invariants $\rho_{\alpha_\pm}$, where $\alpha_\pm(b_1) = \pm 1$, $\alpha_\pm(f) = -1$, and $\alpha_\pm(x) = 1$ for all of the other generators.  As above, each of these extends to a branched cover, and we get $\rho_{\alpha_{\pm}}(Q_{2g,e}) = -e/2$.
\end{proof}

The main result of this section, Theorem~\ref{T:g-sign}, makes use of branched coverings constructed via Proposition~\ref{twist}. Constraints on the Euler class and genus of an essential surface come from Smith-theory bounds on the homology of these branched coverings, and so we assume that the coverings have order a power of $2$.

\begin{theorem}\label{T:g-sign}
Suppose that $F_h \subset W$ is an essential embedding with normal Euler number $e$ in a homology cobordism between $M_0$ and $M_1$.  Let $\alpha\co H_1(M_0) \to U(1)$ be a representation with image $\Z_{2^k}$, and let $\altau$ be the associated twisted representation arising from Proposition~\ref{twist}.  If $k\ge 1$, then
\begin{equation}\label{E:g-sign}
-2^{k} h \leq \rho_{\altau}(M_1) - \rho_\alpha(M_0) + e/2 \leq 2^{k} h
\end{equation}
whereas if $k=0$, then
\begin{equation}\label{E:g-sign0}
-2 h \leq \rho_{\altau}(M_1) - \rho_\alpha(M_0) + e/2 \leq 2 h
\end{equation}
\end{theorem}
\begin{proof}
By construction, $\alpha$ on $H_1(M_0)$ and $\altau$ on $H_1(M_1)$ extend to a representation on $H_1(V)$ whose restriction (still denoted $\altau$) is twisted on $\qhe$.  Note that for $k=0$ (that is, if $\alpha$ is the trivial representation) the image of $H_1(V)$ under $\altau$ is $\Z_2$, but if $k\geq 1$, then the image is $\Z_{2^k}$.  This accounts for the difference between equations~\eqref{E:g-sign} and~\eqref{E:g-sign0}.

 Taking into consideration that $\sign(V) =0$, the Atiyah-Patodi-Singer theorem says that
$$
\rho_{\altau}(M_1) - \rho_\alpha(M_0)  + \rho_\alpha(\qhe) = - \sign_{\altau}(V,M_0)
$$
where $\qhe$ is oriented as part of the boundary of $V$.   Since that orientation of $Q$  is opposite to its orientation as the boundary of $P$, Proposition~\ref{rhoQ} gives
$$
\rho_{\altau}(M_1) - \rho_\alpha(M_0)  + \frac{e}{2} = - \sign_{\altau}(V,M_0)
$$
Now $ \sign_{\altau}(V,M_0)$ is bounded by the rank of $H_2(V,M_0;\C_{\altau})$, which according to~\cite[Proposition 1.4]{gilmer:thesis} is bounded in turn by $2^{k} \rank H_2(V,M_0;\Z_2)$ for $k \geq 1$ and $2 \rank H_2(V,M_0;\Z_2)$ for $k=0$.  A straightforward calculation with the Mayer-Vietoris sequence for $(W,M_0) = (V,M_0) \cup_Q P$ and Lemma~\ref{L:Q} show that $\rank H_2(V,M_0;\Z_2) = h$, which implies the result.
\end{proof}
As in the proof of Theorem~\ref{T:Mbound}, we can get stronger results by varying $\alpha$, although only over representations with image $\Z_{2^k}$.  This includes interchanging the roles of $\alpha$ and $\altau$. For simplicity, we give only the result for $\alpha$ the trivial representation.
\begin{corollary}\label{T:rhobound}
Let $M$ be a rational homology sphere. Suppose that $W$ is a homology cobordism, and that $F_h\subset W$ is essential and has normal Euler number $e$, and twisting $\tau\in H^1(M_1;\Z_2)$.  Let $\alpha$ denote the trivial representation.  Note that $\rho_{\alpha}=0$ and that the homology cobordism invariance of $\rho$ implies that $\rho_{\altau}(M_1) = \rho_{\altau}(M_0)$, so we write $\rho_{\altau}$ for either of these.
 Then
$$h \ge\frac{\left|\rho_{\altau}\right|}{2}.$$
Moreover, $|e| \le 2\left(2h - \left|\rho_{\altau}\right|\right)$.
\end{corollary}
\begin{proof}
Apply Theorem~\ref{T:g-sign} twice, interchanging the roles of $M_0$ and $M_1$, to get
\begin{alignat}{2}
-2h & \leq \rho_{\altau} + \frac{e}{2}\  && \leq 2h \label{m0}\\
-2h & \leq -\rho_{\altau} + \frac{e}{2} &&  \leq 2h \label{m1}
\end{alignat}
The right side of~\eqref{m0} and the left side of~\eqref{m1} give $\rho_{\altau} \leq 2h$, and the other pair give $-\rho_{\altau} \leq 2h$.  Similarly, the right hand sides of~\eqref{m0} and~\eqref{m1} give $e/2 \leq 2h -   \left|\rho_{\altau}\right|$, and the left hand sides give the same upper bound for $-e/2$.
\end{proof}

\subsection{Sample computations}\label{S:examples}
We present a couple of examples to explain how Theorem~\ref{T:g-sign} works and also to contrast its implications with those stemming from Theorem~\ref{T:twist-bound}.
\begin{example}\label{E:drho} According to Corollary~\ref{C:simple} (as explicated in Section~\ref{S:lens}) there is no smooth essential embedding of $\RP^2$ in $L(4,1) \times I$ with any Euler class.  Let $\omega$ be a primitive fourth root of unity, and let $g \in H_1(L(4,1))$ be a generator.  We compute (using~\cite[p. 187]{casson-gordon:orsay}) that for $\alpha_1\co H_1(L(4,1)) \to U(1)$ with $\alpha_1(g) = \omega$ we have $\rho_{\alpha_1}(L(4,1))  = -1/2$. To apply Theorem~\ref{T:g-sign}, note that
$\alpha_1^\tau(g) = -\omega$, so $\rho_{\alpha_1^\tau}(L(4,1))  - \rho_{\alpha_1}(L(4,1))  = 0$.  Hence in writing the bound in the theorem we have $\epsilon = 0$ and Equation~\eqref{E:g-sign} gives
$$
-4 \leq  e/2 \leq 4.
$$
Recalling from Proposition~\ref{P:mod4} that $e \equiv 2\pmod{4}$, we see that $e$ could be $-6,\;-2,\;2$ or $6$.

We obtain a stronger result by considering the representation $\alpha_2$ with $\alpha_2(g) = \omega^2 = -1$, which has $\rho_{\alpha_2} = -1$.  Now $\alpha_2^\tau(g) =1$, so $\rho_{\alpha_2^\tau} = 0$.  Then from Equation~\eqref{E:g-sign}
$$
-2 \leq -1 + e/2 \leq 2
$$
which rules out  $e=-6$.  Similarly, choosing $\alpha$ to be the trivial representation gives $\altau = \alpha_2$, which rules out $e=6$.  The conclusion is that Theorem \ref{T:g-sign} does not obstruct the existence of a locally flat embedding of $\RP^2$ in $L(4,1) \times I$ with Euler class $\pm 2$, although there is no such smooth embedding.
\end{example}

On the other hand, sometimes the $\rho$-invariants give stronger embedding obstructions than the $d$-invariants, as the following example demonstrates.

\begin{example}\label{E:rhod}
Let $Y$ be the Seifert-fibered space $M(-1; (3,2), (4,1), (6,1))$ that also appeared in Proposition~\ref{P:non-simple}, where it was remarked that the maximum $d$-invariant difference $\Delta = 1/2$. Thus, Corollary~\ref{C:rp2} would in principle allow for an embedding of $\RP^2$ in $Y \times I$ with Euler class $0$. On the other hand, $Y$ is also the result of $+6$ surgery on the $(4,3)$ torus knot, and this description allows us to compute $\rho_\alpha(Y) $ for a $U(1)$ representation, via~\cite[Lemma 3.1]{casson-gordon:stanford}.  The formula gives for the non-trivial $\Z_2$ representation $\alpha$ that $\rho_\alpha(Y) = -2 -\sign(T_{4,3}) = 4$, since  the signature of the $(4,3)$ torus knot is $-6$. Applying Equation~\eqref{E:g-sign} with $\alpha$ equal to the trivial and non-trivial $\Z_2$ representations gives (respectively)
\begin{align*}
-4 -2h  &\leq \frac{e}2 \leq -4 + 2h\\
4 -2h  &\leq \frac{e}2 \leq 4 + 2h
\end{align*}
Combining these gives $4 -2h \leq  -4 + 2h$, so $h\geq 2$.   If the Euler class is $0$, then the mod $4$ congruence (Proposition~\ref{P:mod4}) implies that $h$ is odd, so $h \geq 3$. (In fact, the smallest genus that we can find for a surface with $e=0$ is $9$, which suggests that there is room for improvement in our methods.)
\end{example}

\section{Topological embeddings}\label{S:top}
It is well-known that embedding problems for smooth and topological (always meaning locally-flat) surfaces may be very different. In particular, there are smooth $4$-manifolds where the minimal genus of a topologically embedded surface carrying a particular integer homology class is lower than the minimal genus of a smoothly embedded surface in the same homology class.  In this section, we address the analogous question for non-orientable surfaces.  We give an example of a $3$-manifold $M$ and a $\Z_2$ homology class in $H_2(M;\Z_2)$ that is represented by a locally flat $\RP^2$ in a topological $4$-manifold with the homology of $M \times I$, but where the minimal genus for a smoothly embedded representative in a smooth $4$-manifold is $3$.

If $K$ is a knot in $S^3$, let $S^3_r(K)$ denote $r$-framed surgery on $K$.
\begin{theorem}\label{top-emb}
Suppose that the knot $K$ is smoothly (resp. topologically) slice.  Then the non-trivial homology class in $H_2(S^3_2(K);\Z_2)$ is represented by a smooth (resp. locally flat) embedded $\RP^2$ in a smooth (resp.~topological) manifold $W$ with the homology of $S^3_2(K)\times I$ and with  $\partial W = -S^3_2(K) \sqcup S^3_2(K)$.
\end{theorem}
\begin{proof}
Let $C$ be a concordance in $S^3\times [0,1/2]$ between $K$ and the unknot $\OO$. Then, as in~\cite{gordon:contractible}, one can do $+2$ surgery on $C$ to obtain a smooth or topological homology cobordism between $S^3_2(K)$ and $S^3_2(\OO) = L(2,1)$. Now double this homology cobordism along $L(2,1)$ to get the manifold $W$.  Since $L(2,1)$ contains an essential $\RP^2$, the result follows.
\end{proof}

\begin{example}\label{E:top-emb} Let $K$ be the positive-clasped untwisted
Whitehead double of the trefoil $T_{2,3}$, which is topologically slice~\cite{freedman-quinn}. Hence by
Theorem~\ref{top-emb} there is a homology cobordism $W$ between $S^3_2(K)$ and itself that contains a locally-flat essential $\RP^2$. On the other hand, the following remarks will show that there is no smooth essentially embedded $\RP^2$ in any such homology cobordism.  Note that since the genus of $K$ is one, there is an embedded $F_3$ in  $S^3_2(K)$ carrying the non-trivial class in mod $2$ homology.

According to~\cite[Appendix A]{hedden-kim-livingston:2-torsion} (compare~\cite{hedden:double}) the knot Floer chain complex for $K$ is filtered homotopy equivalent to that of $T_{2,3}$ plus an acyclic complex.  The integral surgery formula~\cite{oz:z-surgery} implies that the Heegaard Floer homology for $S^3_2(K)$ is ($\Q$-graded) isomorphic to that of $S^3_2(T_{2,3})$.  One can calculate the $d$-invariants of $S^3_2(T_{2,3})$ via the integer surgery formula or surgery exact sequences (note that $2$ is a `large' surgery since $2 > 2g(T_{2,3}) -1$) or by writing $S^3_2(T_{2,3})$ as a Seifert fibered space and using the algorithm of~\cite{oz:plumbed}.  Either method yields that the two $d$-invariants are $-1/4$ and $-7/4$, so by Corollary~\ref{C:rp2}, there is no smoothly embedded $\RP^2$ in any smooth homology cobordism from $S^3_2(K)$ to itself.
\end{example}

\begin{remark}
Unless $K$ is the unknot, the manifold $W$ constructed in the proof of Theorem~\ref{top-emb} will not be homeomorphic to $M \times I$. For if it were, then $W$ would retract onto $M$; the restriction of this retraction would be a degree-one map $L(2,1) \to M$.  Such a map would be a surjection on $\pi_1$ and hence $M$ would have fundamental group of order $2$.  By Perelman's solution~\cite{perelman:extinction,morgan-tian:poincare} to the Poincar\'e conjecture, $M \cong L(2,1)$, implying~\cite{kmos:lens} that $K$ is the unknot.  Finding an embedding in $M \times I$ seems to be a challenging problem.
\end{remark}

\section{Embeddings in closed manifolds}\label{S:definite}
In this section, we study restrictions on the genus and normal Euler number of a closed, non-orientable surface $F$ embedded in a closed $4$-manifold $X$ with $H_1(X;\Z)=0$. Our first result concerns embeddings in definite manifolds.

\begin{theorem}\label{T:definite}
Suppose $X$ is a closed positive definite $4$-manifold with $H_1(X;\allowbreak\Z)=0$ and $b_2(X)=b$, and $F \subset X$ is a closed non-orientable surface of genus $h$ with normal Euler number $e$. Denote by $\ell$ the minimal self-intersection of an integral lift of $[F]$.
Then
\[
e \equiv \ell - 2h \pmod 4 \quad \text{and} \quad e \ge \ell - 2h.
\]
Additionally, if $\ell = b$, then
\[
e \le 9b + 10h - 16.
\]
\end{theorem}

In the case where $h=1$, the first part of the theorem is a result of Lawson \cite{lawson:rp2}, and the second part follows from a theorem of Ue \cite[Theorem 3]{ue:spherical}. Moreover, the case where $b=0$ was proven by Massey~\cite{massey:whitney}, verifying Whitney's conjecture~\cite{whitney:manifolds-lectures} on normal bundles for surfaces in $S^4$ (or more generally, any homology sphere):

\begin{corollary}\label{C:sphere}
Suppose $X$ is a homology $4$-sphere and $F \subset X$ is a closed non-orientable surface of genus $h$ with normal Euler number $e$. Then
$$e \equiv  2h \pmod 4 \quad \text{and}\quad |e| \le 2h.$$
\end{corollary}

\begin{proof}
Apply Theorem \ref{T:definite} to $X$ with either orientation.
\end{proof}

The first part of Theorem \ref{T:definite} follows from Theorem \ref{T:negative} (Ozsv\'ath and Szab\'o's inequality for correction terms) combined with our computation of the correction terms for the circle bundles $Q_{h,e}$ (Theorem \ref{T:d}). The second part is a special case of a more general theorem concerning non-orientable surfaces that are characteristic for the intersection form, which follows from Rohlin's theorem on the signature of a spin 4-manifold and Furuta's 10/8 theorem.

\begin{theorem}\label{T:spin}
Suppose $X$ is a closed, non-spin $4$-manifold with $H_1(X;\Z)=0$ and $\sigma(X)\ge 0$. Let $F \subset X$ be a closed, non-orientable surface of genus $h$ with normal Euler number $e$ that is characteristic (i.e., $[F] = \PD(w_2(X))$). Then for some $k\in\{0,1,\ldots,h\}$, if we set
\[
e' = e+ 2h - 4k,
\]
the following hold:
\begin{gather}
\label{eq:e'congruence}  e' \equiv \sigma(X) \pmod {16}, \\
\label{eq:e'upper} e'\le \sigma(X)+8(b_+(X)+h-2), \\
\label{eq:e'lower} e'\ge \sigma(X)-8(b_-(X)+h-2)\ \text{if}\ e'<0, \\
\intertext{and}
\label{eq:e'=0} 0 \ge \sigma(X)-8(b_-(X)+h-1)\ \text{if}\ e'=0.
\end{gather}
\end{theorem}

As an immediate consequence of Theorem \ref{T:spin}, we have:

\begin{corollary}\label{C:spin}
Under the assumptions of Theorem \ref{T:spin}, we have
\begin{gather*}
e \equiv \sigma(X) + 2h \pmod 4 \\
\intertext{and}
\min\{-2h,\sigma(X)-8(b_-(X)-2)-10h\} \le  e \le \sigma(X) + 8(b_+(X)-2)+10h.
\end{gather*}
\end{corollary}
Note that the congruence in the corollary could also be deduced from the extension by Guillou and Marin~\cite{guillou-marin:rohlin}  of Rochlin's theorem.  Indeed, the proof of~\eqref{eq:e'congruence}, from which this congruence is deduced, could be adapted to give a proof of the Guillou--Marin result, along the lines of~\cite{matsumoto:rohlin}.

For the proof of Theorem \ref{T:definite}, let $P = P_{h,e}$ be a regular neighborhood of the surface $F \subset X$, let $Q = Q_{h,e} = \partial P$, and let $V$ be the closure of the complement of $P$. As a preliminary step we need to understand how spin$^c$ structures on $V$ restrict to $Q$.

\begin{lemma}\label{L:closed}
Let $X$ be a closed 4-manifold with $H_1(X;\Z)=0$ and $F \subset X$ a closed non-orientable surface.
If $[F]$ is non-zero in $H_2(X;\Z_2)$, then $H_1(V;\Z)=0$ and the restriction homomorphism $H^2(V;\Z) \to H^2(Q;\Z)$ is surjective. If $[F]$ is trivial in $H_2(X;\Z_2)$, then $H_1(V;\Z)\cong \Z_2$ and the cokernel of the restriction homomorphism $H^2(V;\Z) \to H^2(Q;\Z)$ is isomorphic to $\Z_2$.
\end{lemma}

\proof
Assume first $0\ne [F]\in H_2(X;\Z_2)$. Consider the exact sequences for $X$ and $F$ corresponding to the coefficient sequence $0 \to \Z \to \Z \to \Z_2 \to 0$. Since $H^2(X;\Z_2) \to H^2(F;\Z_2)$ is onto, so is $H^2(X;\Z) \to H^2(F;\Z)$. From the exact sequence of the pair $(X,F)$ it now follows that $H^3(X,F;\Z)=0$. Using excision we get $H^3(V,Q;\Z)=0$ (which implies the surjectivity of the restriction homomorphism) and finally using Poincar\'e--Lefschetz duality $H_1(V;\Z)=0$.

If $0= [F]\in H_2(X;\Z_2)$ then $H^3(X,F;\Z)\cong \Z_2$ from which the result follows in this case.
\endproof

\proof[Proof of Theorem \ref{T:definite}]
Recall that by Donaldson's diagonalization theorem the intersection form of $X$ is diagonal. Denote by $x_1,\ldots,x_b$ a basis for $H_2(X;\Z)$ with $x_i \cdot x_j = \delta_{ij}$ for all $i,j$. We may relabel the generators so that the $\Z_2$ homology class of $F$ is equal to $[F]=\sum_{i=1}^\ell \bar x_i$, where $\bar{x}$ denotes the reduction of an integral homology class $x$ modulo 2. Assume first $\ell>0$. Let $\xi\in H^2(V;\Z)$ be the image of $\sum_{i=\ell+1}^b x_i^*$, where $x^*$ denotes the hom-dual of a homology class $x$. Note that $\xi$ is characteristic, but does not come from a characteristic element on $X$, hence the corresponding spin$^c$ structure $\spincs$ on $V$ does not extend over $X$. In particular the restriction $\spinct$ of $\spincs$ to $Q$ is a torsion spin$^c$ structure that does not extend over $P$. Since $-V$ is a negative definite manifold with boundary $Q$ and $H^1(V;\Z)=0$ it follows from Theorem \ref{T:negative} that
$$4\dbot(Q,\spinct)\geq \ell -2(h-1),$$
where the two sides of the inequality are congruent modulo 8. Using the values of $\dbot$ from Theorem \ref{T:d} this yields the first two conditions of the theorem. The last one follows from Corollary \ref{C:spin} after substitution $\sigma(X)=b$.

If $\ell=0$ then by Lemma \ref{L:closed} $H^2(V;\Z)$ contains a nontrivial element of order 2 by which we can twist the spin$^c$ structure determined by $\xi$ on $X$. Then the same argument as above applies.
\endproof

Recall that $Q_{h,e}$ may be obtained as a surgery on a knot in $\#^h S^1 \times S^2$, namely the connected sum of $h$ copies of a knot in $S^1 \times S^2$ representing twice a generator of the first homology, as drawn below for $h=3$.
\begin{figure}[h]
  \begin{center}
    \labellist
    \normalsize\hair 0mm
    \pinlabel {{\footnotesize$0$}} at 28 75
    \pinlabel {{\footnotesize$0$}} at 99 75
    \pinlabel {{\footnotesize$0$}} at 165 75
    \pinlabel {{\footnotesize$e+2h$}} at 213 26
    \endlabellist
    \includegraphics[scale=.85]{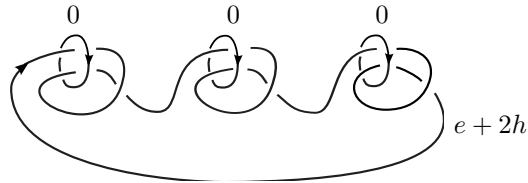}
  \end{center}
  \caption{Kirby diagram for $Q_{h,e}$, shown when $h=3$.}
  \label{F:Q3e}
\end{figure}

Starting from the surgery manifold we describe below spin 4-manifolds whose spin boundary is $Q_{h,e}$, equipped with one of the $2^h$ spin structures that do not extend over the disk bundle $P_{h,e}$. Using these in conjunction with Furuta's 10/8 Theorem~\cite{furuta:118} we get bounds of Theorem \ref{T:spin}.

\begin{lemma}\label{L:spin}
For each spin structure $\spincs$ on $Q_{h,e}$ that does not extend over the disk bundle $P_{h,e}$, there exist an integer $k \in \{0, 1, \dots, h\}$ and a spin 4-manifold $Z$ with spin boundary $(Q,\spincs)$ such that if we define $e' = e+2h-4k$, then $\sigma(Z) = -e'$ and
\[
b_2(Z) = \begin{cases}
h-1+|e'| & e' \ne 0 \\
h+1 & e'= 0.
\end{cases}
\]
\end{lemma}

\proof
Let $Z_0$ be the $4$-manifold specified by the Kirby diagram as in Figure~\ref{F:Q3e}. We label the $(e+2h)$-framed component $K$ and the $0$-framed components $C_1, \dots, C_h$. Note that every component, as drawn in $S^3$, is an unknot. We orient the link components such that $\lk(K, C_i) = 2$ for $i=1,\ldots,h$. Note that $b_2(Z_0)=h+1$ and $\sigma(Z)=0$. The boundary of $Z_0$ is $Q = Q_{h,e}$. Replacing each $C_i$ with a dotted circle (for a 1-handle addition) yields a Kirby diagram for the disk bundle $P = P_{h,e}$.

Recall that spin structures on $Q_{h,e}$ are in one-to-one correspondence with characteristic sublinks of $\{K,C_1,\ldots,C_h\}$; see \cite{kaplan:even}, \cite[Section 5.7]{gompf-stipsicz:book}; the $2^h$ spin structures on $Q$ that do not extend over $P$ correspond to the sublinks that include $K$. When $e$ is odd, these are the only spin structures; when $e$ is even, every sublink is characteristic, and the empty sublink corresponds to the restriction of the unique spin structure on $Z_0$.

Up to reindexing, we may assume that the characteristic sublink corresponding to the given spin structure $\spincs$ is $\{K, C_1, \dots, C_k\}$, where $k \in \{0, \dots, h\}$. We obtain a new Kirby diagram for $Z_0$ by handle-sliding $K$ over $C_i$ for each $i=1, \dots, k$, as shown in Figure \ref{F:slide}, to obtain a new knot $K'$ with framing $e' = e + 2h - 4k$. Note that $K'$ remains unknotted in $S^3$. By \cite[Theorem 5.7.14]{gompf-stipsicz:book}, the characteristic sublink of the new diagram corresponding to $\spincs$ is $\{K'\}$.

\begin{figure}[h]
  \begin{center}
    \labellist
    \normalsize\hair 0mm
    \pinlabel {{\footnotesize$0$}} at 57 155
    \pinlabel {{\footnotesize$0$}} at 230 155
    \pinlabel {{\footnotesize$0$}} at 373 155
    \pinlabel {{\footnotesize$e+2h-4k$}} at 475 52
    \endlabellist
    \includegraphics[scale=.45]{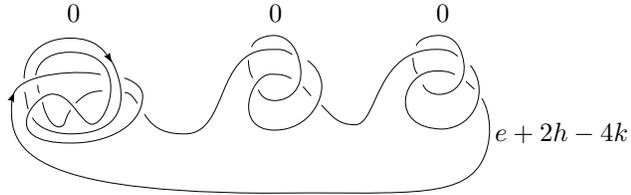}
  \end{center}
  \caption{New Kirby diagram for $Q_{h,e}$ after performing $k$ handle-slides to change the characteristic sublink, shown when $h=3$ and $k=1$.}
  \label{F:slide}
\end{figure}

We now perform a sequence of blow-ups and blow-downs to change the $4$-manifold. Specifically, if $e'\ne 0$ let $\veps=\pm 1$ be the sign of $e'$ and blow up $|e'|-1$ $(-\veps)$-framed meridians of $K'$. The resulting framing of $K'$ is $\veps$, and since $K'$ is an unknot, we may blow it down to obtain a Kirby diagram for a spin manifold $Z$ with $b_2(Z) = h-1+|e'|$ and $\sigma(Z)=-e'$. In the new diagram, the empty sublink corresponds to $\spincs$, meaning that $Z$ has spin boundary $(Q, \spincs)$, as required. Likewise, if $e'=0$, we blow up a $(+1)$-framed meridian of $K'$ to change the framing on $K'$ to $+1$ and then blow down $K'$ as before to obtain $Z$ with $b_2(Z)=h+1$ and $\sigma(Z)=0$.
\endproof

\begin{proof}[Proof of Theorem \ref{T:spin}]
Let $V$ be the closure of the complement of the disk bundle $P$ over $F$ in $X$. It follows from the Mayer-Vietoris sequence for $X=P \cup_Q V$ that $b_1(V)=0$, $b_2(V)=b_2(X)+h-1$, and that $H_2(Q;\Z)$ injects into $H_2(V;\Z)$. Since $F$ is characteristic, $V$ is spin and it induces a spin structure $\spincs$ on its boundary $Q$. By assumption $X$ is not spin, so $\spincs$ does not extend over $P$. Let $k$, $Z$, and $e'$ be as in Lemma \ref{L:spin}.
Define $X'=Z \cup_Q V$, which is a closed spin 4-manifold with $b_1(X')=0$ and $\sigma(X') = \sigma(X) - e'$. Rohlin's theorem implies that
\[
\sigma(X) \equiv e' \pmod {16}.
\]
Also, Furuta's 10/8 theorem states that
\[
4b_2(X')\ge 5|\sigma(X')|+8,
\]

Suppose first that $e' \ne 0$. Then
\[
b_2(X')=b_2(X)+2(h-1)+|e'|.
\]
so
\begin{equation} \label{eq:furuta,e'ne0}
4 b_2(X) + 8(h-1) + 4|e'| \ge 5|\sigma(X) - e'| + 8.
\end{equation}
Recall that $\sigma(X) \ge 0$ by assumption. If $e' \ge \sigma(X)$, then \eqref{eq:e'upper} follows from \eqref{eq:furuta,e'ne0}; if $e' < \sigma(X)$, then \eqref{eq:e'upper} is automatic since $b_+(X) \ge 1$ and $h \ge 1$.
Moreover, if $e' < 0$, then \eqref{eq:e'lower} also follows from \eqref{eq:furuta,e'ne0}.

Similarly, if $e' = 0$, then
\[
b_2(X')=b_2(X)+2h \quad \text{and} \quad \sigma(X')= \sigma(X),
\]
which yields \eqref{eq:e'=0}, as required.
\end{proof}

\def\cprime{$'$}
\providecommand{\bysame}{\leavevmode\hbox to3em{\hrulefill}\thinspace}

\vfil\eject
\appendix
\section{Differences of correction terms for lens spaces \\ By Ira M. Gessel}


\newcommand{\sumx}[2]{\sum_{\substack{#1^{#2}=1\\ #1\ne 1}}}
\newcommand{\sumy}[2]{\sum_{#1^{#2}=-1}}
\newcommand{\z}{\zeta}
\newcommand{\floor}[1]{\left\lfloor#1\right\rfloor}
\newcommand{\ceil}[1]{\left\lceil#1\right\rceil}

\def\noqed{\renewcommand{\qedsymbol}{}}

\theoremstyle{remark}
\newtheorem*{note}{Note}

\makeatletter{\renewcommand*{\@makefnmark}{}
\footnotetext{This work was partially supported by a grant from the Simons Foundation (\#229238 to Ira Gessel).}\makeatother}

\subsection{Definitions}
\label{s-1}

Let $k$ be a positive integer, let $i$ be an arbitrary integer, and  let $q$ be an integer relatively prime to $2k$. Then we
define $g(2k,q,i)$ by
\begin{equation}
\label{e-gk}
g(2k, q,i) = \frac12 (2k-1) -\frac 1k \sum_{j=0}^{k-1}[i+qj],
\end{equation}
where $[m]$ is the least nonnegative residue of $m$ modulo $2k$. Equation \eqref{E:d-shift} expresses $g(2k,q,i)$ as a difference of $d$-invariants of lens spaces.
Let $G(2k,q)=\max_i g(2k,q,i)$.

We note that $G(2,1)=1/2$.
We will show that  for $q>0$, $G(2k,q)$ satisfies the same recurrence as that given for $N(2k,q)$ in section
\ref{S:lens}:
\begin{equation}
\label{e-req1}
G(2k,q)=G(2(k-q),q') + \tfrac12,
\end{equation}
where $q'\equiv\pm q\pmod{2(k-q)}$ and $1\le q'\le k-q$.

We will prove a slightly more general recurrence from which \eqref{e-req1} follows easily when combined with the observation that $G(2k, q)$ depends only on the residue of $q$ modulo $2k$:
\begin{proposition}
\label{p-1}
Let $k$ and $q$ be   positive integers with $q$ relatively prime to $2k$.  Then
\begin{equation}
\label{e-G}
G(2k+2q,q) = G(2k, q)+\tfrac12=G(2k,-q)+\tfrac12.
\end{equation}
\end{proposition}

We will derive Proposition \ref{p-1} from similar properties of  $g(2k,q,i)$. From now on we always assume that $k>0$ and that $2k$ and $q$ are relatively prime (and thus $q$ must be odd).

The (easy) proof of the second equality in \eqref{e-G} will be given after Lemma \ref{a2}. The key to proving the first equality is the formula
\begin{equation}
\label{e-g2}
g(2k+2q, q,i) = g(2k, q, i)+\tfrac12,
\end{equation}
where $k+q>0$ and $-k\le i<k+q$.

We will see that if $q>0$, then the range of validity of \eqref{e-g2} includes a maximizing value of $i$ for each side, and thus  $G(2k+2q,q) = G(2k,q)+\tfrac12$ follows. Our proof uses a formula (equation \eqref{e-w1}) for $g(2k,q,i)$ as a sum involving roots of unity, similar to a Dedekind sum. In section \ref{s-another} we describe another proof of the the recurrence \eqref{e-g2} based on
another formula for $g(2k,q,i)$ (equation \eqref{e-gx}) as a sum of powers of $-1$.

Expressions for the $d$-invariants of lens spaces in terms of generalized Dedekind sums have also been given by Jabuka, Robins, and Wang \cite{jrw} and by Tange \cite{tange}.

\subsection{Basic properties}

It is clear that $g(2k,q,i+2k)=g(2k,q+2k,i)=g(2k, q, i)$ for all $i$ and $q$, so for fixed $k$, $g(2k, q, i)$ is a periodic function of both $q$ and $i$ of period $2k$.
We can easily obtain two formulas for $g(2k, q, i+q)$ directly from the definition.

\begin{lemma}\label{a2}
For all \textup(allowable\textup) $k$, $q$, and $i$ we have
\label{l-1x}
\begin{gather}
\label{e-1.1}
g(2k, -q, i) = g(2k,q, i+q+k)\\
\label{e-1.2}
g(2k,q, i) - g(2k,q,i+q) =
\left.\begin{cases}
   \phantom{-}1, &\text{if $0\le [i] < k$}\\
   -1, &\text{if $k\le [i] <2k$}
\end{cases}\right\}
  =(-1)^{\floor{i/k}}.
\end{gather}
\end{lemma}
\begin{proof}
For \eqref{e-1.1}, we reverse the order of the sum \eqref{e-gk} defining $g(2k, -q,i)$, so
\begin{align*}
g(2k, -q,i) &= \frac12(2k-1) -\frac1k \sum_{j=0}^{k-1} [i-q(k-1-j)]\\
  &=g(2k, q, i+q-qk) = g(2k, q, i+q+k),
\end{align*}
since $i+q-qk\equiv i+q+k \pmod{2k}$.

Equation \eqref{e-gk} gives
\(
g(2k,q,i) - g(2k,q,i+1) = \frac1k\bigl( [i+kq] - [i]\bigr)
\)
from which \eqref{e-1.2} follows immediately.
\end{proof}

From \eqref{e-1.1} we get $G(2k,-q)=G(2k,q)$, the second equality of \eqref{e-G}.

By \eqref{e-1.2},   the maximum value $\max_i g(2k,q,i)$ cannot occur at $i=q, q+1,\dots,k+q-1$.
We will see in the next section that if $q>0$ then the maximum of $g(2k,q,i)$ for $0\le i< 2k$ can only occur for $0\le i <q$, and must occur for some $i$ with $0\le i\le (q-1)/2$.

\subsection{Sums over roots of unity}

In this section we show that $g(2k,q,i)$ can be expressed as a sum involving roots of unity, similar to a Dedekind sum.

The following two formulas are straightforward partial fraction expansions for rational functions of $x$.
\begin{lemma}
\begin{align}
\label{e-pf0}
\sum_{\zeta^k=-1}\frac{\zeta^{i+1}}{\zeta-x} &= \frac{kx^{i}}{1+x^k},
  \text{ for $0\le i < k$}\\
\label{e-pf1}
\sum_{\zeta^{2k}=1} \frac{\zeta^{i+1}}{\zeta-x}&=\frac{2kx^{i}}{1 -x^{2k}},
  \text{ for $0\le i<2k$}.
\end{align}
\end{lemma}

Our next formula is well-known and can be proved in many ways; the proof we give here is essentially that of \cite[Corollary 3.2]{g-gended}

\begin{lemma}
Let $m$ be a positive integer. Then for any integer $i$,
\label{l-eis2}
\begin{equation*}
[i]= \frac12 (2k-1) - \sum_{\substack{\zeta^{2k}=1\\ \zeta\ne 1}}
\frac{\zeta^{i+1}}{\zeta-1}.
\end{equation*}
\end{lemma}

\begin{proof}
Without loss of generality, we may assume that $0\le i < 2k$. Then
 we have
\begin{align*}
\sum_{\substack{\zeta^{2k}=1\\ \zeta\ne 1}} \frac{\zeta^{i+1}}{\zeta-1}
   &= \lim_{x\to1}\biggl(\sum_{\zeta^{2k}=1} \frac{\zeta^{i+1}}{\zeta-x}-
   \frac{1}{1-x}\biggr)\\
   &=\lim_{x\to1}\left(\frac{2k x^i}{1-x^{2k}} -\frac{1}{1-x}\right),
      \text{ by  \eqref{e-pf1},}\\
   &=\lim_{x\to 1}\left(\frac{{2k}x^i -(1+x+\cdots + x^{{2k}-1})}{1-x^{2k}}\right)\\
   &=\frac{{2k}i -(1+2+\cdots +{2k}-1)}{-{2k}}, \text{ by l'H\^opital's rule,}\\
   &=\frac{1}{2}({2k}-1) -i.\qedhere
   \end{align*}
\end{proof}

Next, we prove a useful formula for $g(2k,q,i)$.

\begin{theorem}
\label{t-1}
\begin{equation}
\label{e-w1}
g(2k,q,i) =-\frac 2k \sum_{\zeta^k = -1}
   \frac{\zeta^{i+1}}{(\zeta-1)(\zeta^q-1)}.\\
\end{equation}
\end{theorem}

\begin{proof}
Applying Lemma \ref{l-eis2} to \eqref{e-gk}
gives
\begin{align}
g(2k,q,i) &=\frac12(2k-1) -\frac1k\sum_{j=0}^{k-1}\biggl( \frac12(2k-1)
 - \sumx{\z}{2k} \frac{\z^{i+qj+1}}{\z-1}\biggr)\notag\\
&=-\frac{1}{k} \sum_{\substack{\zeta^{2k}=1\\ \zeta\ne 1}}\frac{\zeta^{i+1}}{1-\zeta}\cdot\frac{1-\zeta^{qk}}{1-\zeta^{q}}.\label{e-gsum}
\end{align}
Now if $\zeta^{2k}=1$ then $\zeta^k$ is either 1 or $-1$. If $\zeta^k=1$ then  $1-\z^{qk}=0$ and if
$\z^k=-1$ then (since $q$ is odd) $1-\z^{qk}=2$. Thus \eqref{e-w1} follows.
\end{proof}

We can also use this computation to find a simpler formula for $g(2k,q,i)$ that we will use in section \ref{s-another} (though we will give an independent proof of this formula there).
 First we note that setting $x=1$ in \eqref{e-pf0} gives
\begin{equation*}
\sum_{\zeta^k=-1}\frac{\zeta^{i+1}}{\zeta-1} = \frac{k}{2}
\end{equation*}
for $0\le i<k$ and it follows that, for any $i$,
\begin{equation}
\label{e-z1}
\sum_{\zeta^k=-1}\frac{\zeta^{i+1}}{\zeta-1} = (-1)^{\floor {i/k}}\frac{k}{2}.
\end{equation}

Then since the only nonzero terms in \eqref{e-gsum} are those with $\zeta^k=-1$, we have
\begin{equation}
\label{e-gx}
\begin{aligned}
g(2k,q,i) &= -\frac{1}{k} \sumy \z k\frac{\zeta^{i+1}}{1-\zeta}\cdot\frac{1-\zeta^{qk}}{1-\zeta^{q}}\\
    &= -\frac{1}{k} \sumy \z k\sum_{j=0}^{k-1}\frac{\zeta^{i+qj+1}}{1-\zeta}\\
    &=  \frac{1}{2}\sum_{j=0}^{k-1} (-1)^{\floor{(i+qj)/k}}, \text{ by \eqref{e-z1}.}
\end{aligned}
\end{equation}

We have a few easy consequences of Theorem \ref{t-1}.

\begin{corollary}
\label{c-1}
We have
\begin{gather}
\label{e-a61}
g(2k,q,i+k)  = -g(2k,q,i)\\
\label{e-a62}
g(2k,q,i) = g(2k,q,q-1-i)
\end{gather}
\end{corollary}
\begin{proof}
Formula \eqref{e-a61} is immediate from Theorem \ref{t-1}. For \eqref{e-a62}, we replace $\z$ with $\z^{-1}$ in \eqref{e-w1}, getting
\begin{align*}
g(2k,q,i) &=-\frac 2k \sum_{\zeta^k = -1}
   \frac{\zeta^{-i-1}}{(\zeta^{-1}-1)(\zeta^{-q}-1)}\\
   &=-\frac2k \sum_{\zeta^k = -1}
      \frac{\zeta^{q-i}}{(1-\z)(1-\z^q)}=g(2k,q,q-1-i).\qedhere
\end{align*}
\end{proof}

We note that \eqref{e-a62} is equivalent to the statement that for fixed  $k$ and $q$, $g(2k,q, i+(q-1)/2)$ is an even function of $i$.

Next we find an interval containing the value of $i$ that maximizes $g(2k, q,i)$.

\begin{lemma}
\label{l-max}
For $q>0$,
there is some $i$ with $0\le i \le (q-1)/2$ satisfying $g(2k,q,i)=G(2k,q)$.
\end{lemma}

\begin{proof}

As we noted earlier, by  \eqref{e-1.2} the maximum value $\max_i g(2k,q,i)$ cannot occur at $i=q, q+1,\dots,k+q-1$.
But by \eqref{e-a62},  $g(2k,q, i) = g(2k,q, q-1-i) = g(2k, q, q-1-i+2k)$ so the maximum cannot occur for $q\le q-1-i+2k\le k+q-1$, which is equivalent to $k\le i\le 2k-1$.  Thus if $0<q<2k$ (wich we may assume without loss of generality) then any  $i$ in $\{0,\dots 2k-1\}$ for which $g(2k,q,i)$  attains its maximum must have $0\le i <q$. Moreover, by \eqref{e-a62} again, since $g(2k,q,i) = g(2k,q, q-1-i)$, if $0\le i < q$ then there is at least one $i$ for which $g(2k,q,i)$ attains its maximum satisfying $0\le i\le (q-1)/2$.
\end{proof}

We now prove a fundamental recurrence for $g(2k,q,i)$.
\begin{theorem}
\label{t-rec}
Suppose that $k+q>0$ and that $-k\le i < k+q$. Then
\begin{equation}
\label{e-rec}
g(2k+2q, q,i) =g(2k, q, i)+\tfrac12.
\end{equation}
\end{theorem}
\begin{proof}
Since $\z^k=-1$, we can write \eqref{e-w1} as
\begin{equation}
\label{e-w2}
g(2k,q,i) =\frac 2k \sum_{\zeta^k = -1}
   \frac{\zeta^{i+1}}{(\zeta-1)(\zeta^{k+q}+1)}.
\end{equation}

We first consider  the case in which  $-1\le i < k+q$, so the summand in \eqref{e-w2}
is a proper rational function of $\zeta$, and therefore has the partial fraction expansion
\begin{equation}
\label{e-pf3}
\frac{\zeta^{i+1}}{(\zeta-1)(\zeta^{k+q}+1)}
=\frac{1/2}{\zeta-1} -\frac 1{k+q} \sum_{\eta^{k+q}=-1}\frac{\eta^{i+2}}{(\eta-1)(\zeta-\eta)}.
\end{equation}
By  \eqref{e-pf0} with $i=k-1$ we have
\begin{equation}
\label{e-pf2}
\sum_{\zeta^k=-1}\frac{1}{\zeta-\eta} = -k\frac{\eta^{k-1}}{\eta^k+1}.
\end{equation}

Applying \eqref{e-pf3} and \eqref{e-pf2} to \eqref{e-w2} gives
\begin{equation*}
 g(2k,q,i) = -\frac12 +\frac 2{k+q} \sum_{\eta^{k+q}=-1}\frac{\eta^{i+k+1}}{(\eta-1)(\eta^k+1)}
\end{equation*}
and by \eqref{e-w2}, this is equal to
 $ -1/2+g(2k+2q, -q, i+k)$.

Now we have
\begin{align*}
g(2k+2q, -q,i+k) &= -g(2k+2q, -q, i-q), \text{ by \eqref{e-a61}},\\
    &=-g(2k+2q, q,i+q+k), \text{ by \eqref{e-1.1}},\\
          &= g(2k+2q, q,i), \text{ by \eqref{e-a61}},
\end{align*}
and the conclusion follows for $-1\le i < k+q$.

 If $-k\le i <-1$ then the partial fraction expansion for the left side of \eqref{e-pf3} will have additional terms of the form
$c_l\zeta^{-l}$, where $0<l<k$, but for these values of $l$,
$\sum_{\z^k=-1} \z^{-l} = 0$, so the  formula still holds.
\end{proof}


Since the two sums  in the proof of Theorem \ref{t-rec} switch $k$ with $k+q$, we get an equivalent version of the recurrence in the form of a traditional reciprocity theorem by taking  our parameters to be $k$ and $j=k+q$.

\begin{theorem}
\label{t-recip1}
Let $h(k,j,i)= g(2k, j-k,i)$, where $k$ is positive, and  $j$ and $k$ are relatively prime and of opposite parity. Then for $0\le i < j+k$ we have
\begin{equation}
\label{e-recip}
h(j,k,i) + h(k,j,i) = \frac12.
\end{equation}
\end{theorem}

\begin{proof}
Setting $q=j-k$ in
the identity $g(2k, q, i) = -1/2 + g(2k+2q, -q, i+k)$ given in
the proof of Theorem \ref{t-rec} yields
$h(j,k,i+k) - h(k,j,i) = 1/2$ for $-k\le i < j$. By \eqref{e-a61}, $h(k,j,i) = -h(k,j,i+k)$, so we have
\begin{equation*}
h(j,k,i+k) + h(k,j,i+k) = \frac12
\end{equation*}
for $-k\le i < j$. Replacing $i$ with $i-k$ gives \eqref{e-recip}.
\end{proof}

%

We can now finish the proof of Proposition \ref{p-1}. All we need to prove is that for $q>0$,
$G(2k+2q,q) =G(2k,q)+1/2$. This follow Theorem \ref{t-rec} and Lemma \ref{l-max}, since  \eqref{e-rec} is valid for $0\le i \le (q-1)/2$.

There is also a reciprocity form of the recurrence for $G$.  Let us define $H(k,j) = G(2k, k-j)$. Then we have
\begin{equation*}
H(k,j) - H(j,k) = \frac12, \text{ for $k>j>0$}.
\end{equation*}

\begin{note}
Although \eqref{e-rec} holds for $q$ negative as long as $q > -k$, it is generally not true that $G(2k+2q,q)= G(2k,q)+1/2$ for $q<0$. For example, if $k=3$ and $q=-1$ then we have
\[g(4,-1,i) = g(6,-1,i)+\tfrac12, \text{ for $-3\le i<2$,}\]
by Theorem \ref{t-rec}, but the maximum value of $g(6,-1,i)$ occurs for $i=2$, and also for $i=-4$. (Lemma \ref{l-max} guarantees us that the maximum value of $g(6,-1,i)=g(6,5,i)$ occurs for some $i$ with $0\le i\le (5-1)/2=2$.)
\end{note}

It is not hard to compute an $i$ for which $G(2k,q) = g(2k,q,i)$.
\begin{proposition}
Let us define $I(2k,q)$ for $q>0$ by
\begin{equation*}
I(2k,q) = \begin{cases}
      0, &\text{ if $k=q=1$,}\\
      I(2k, [q]), &\text{ if $q>2k$,}\\
      I(2k-2q, q), &\text{ if $q<k$,}\\
      I(2k,2k-q) +q- k, &\text{ if $k<q<2k$.}
\end{cases}
\end{equation*}
Then $G(2k,q) = g(2k, q, I(2k,q))$.
\end{proposition}

\begin{proof}
An easy induction on $q$ shows that that $0\le I(2k,q) \le (q-1)/2$. We now prove the result by induction on $k+q$. The base case, $k=q=1$, is clear, so suppose that $k+q>2$ and that the result holds for $I(2k', q')$ with $k'+q'<k+q$.  Without loss of generality, we may assume that $1<q<2k$.

First suppose that $q<k$. Then by Theorem \ref{t-rec},   for $i<k$ we have
\(
g(2k,q,i) = g(2k-2q,q, i) +1/2
\).
So by Lemma \ref{l-max},
\begin{align*}
G(2k,q) &=\max_{0\le i\le (q-1)/2} g(2k,q,i) \\
  &= \max_{0\le i\le (q-1)/2} g(2k-2q,q,i) +\tfrac12\\
  &=g(2k-2q,q, I(2k-2q,q)) +\tfrac12  \\
  &=g(2k, q, I(2k-2q,q)).
\end{align*}

Next, suppose that $k<q<2q$. Then
\(
g(2k, q, i+q-k) = g(2k, 2k-q, i)
\) by \eqref{e-1.1},
so
\begin{align*}
G(2k, q) &= \max\nolimits_i g(2k, q, i+q-k)\\
  &=\max\nolimits_i g(2k, 2k-q, i)\\
  &=g(2k, 2k-q, I(2k, 2k-q))\\
  &=g(2k, q, I(2k, 2k-q)+q-k).\qedhere
\end{align*}
\end{proof}

\subsection{Another approach}
\label{s-another}

We now describe another approach to the fundamental recurrence \eqref{e-rec} that avoids the use of roots of unity, and gives a more general result. We start by giving a direct proof of equation \eqref{e-gx},
$g(2k,q,i) =  \frac{1}{2}\sum_{j=0}^{k-1} (-1)^{\floor{(i+qj)/k}}.$
By \eqref{e-1.2}, we have
\begin{align}
\sum_{j=0}^{k-1} (-1)^{\floor{(i+qj)/k}}
  &= \sum_{j=0}^{k-1} \bigl(g(2k,q,i+jq) - g(2k,q, i+(j+1)q)\bigr)\notag\\
  &=g(2k, q,i) - g(2k,q, i+kq)\notag\\
  &=g(2k,q,i) - g(2k,q, i+k) \notag
\end{align}
since $g(2k, q, i+2k)=g(2k,q,i)$  and $q$ is odd, so to prove \eqref{e-gx}, it suffices to show that $g(2k,q,i+k) = -g(2k, q, i)$ (which we proved in Corollary \ref{c-1} using roots of unity).
By \eqref{e-gk} we have
\begin{equation}
\label{e-f3}
g(2k, q,i) + g(2k,q, i+k) = 2k-1 -\frac 1k  \sum_{j=0}^{k-1}[i+qj]
 -\frac 1k \sum_{j=0}^{k-1}[i+k+qj].
\end{equation}
The numbers $i+qj$, as $j$ runs from 0 to $2k-1$, run through a complete residue system modulo $2k$, so
\begin{equation}
\label{e-f4}
\sum_{j=0}^{2k-1}[i+qj]=2k(2k-1)/2=k(2k-1).
\end{equation}
But
\begin{align}
\sum_{j=0}^{2k-1}[i+qj]&=\sum_{j=0}^{k-1}[i+qj]
+\sum_{j=0}^{k-1}[i+q(j+k)]\notag\\
&=\sum_{j=0}^{k-1}[i+qj]+\sum_{j=0}^{k-1}[i+qj+k],
\label{e-f5}
\end{align}
since $q$ is odd and thus $qk\equiv k \pmod{2k}$. Then from  \eqref{e-f3}, \eqref{e-f4}, and \eqref{e-f5}, it follows that $g(2k,q,i)+g(2k,q,i+k)=0$, and this completes the proof of \eqref{e-gx}.

Now let us define Laurent polynomials $P(2k,q,i)$  in $u$ by
\begin{equation}
\label{e-P}
P(2k,q,i)  = \sum_{j=0}^{k-1}u^{\floor{(i+qj)/k}}.
\end{equation}

Then we have the following generalization of Theorem \ref{t-rec}, to which it reduces for $u=-1$. We assume that $k>0$ but $q$ need not be relatively prime to $2k$.
\begin{theorem}
\label{t-Prec}
Suppose that $k+q > 0$ and that $-k\le i < k+q$. Then
\begin{equation}
\label{e-rec2}
P(2k+2q,q,i) - P(2k,q,i)=\frac{1-u^q}{1-u}.
\end{equation}
\end{theorem}
\begin{proof}
Let us first take $q$ to be positive.
We define the formal power series $R(2k, q,i)$ by
\begin{equation*}
R(2k, q,i) = \sum_{j=0}^{\infty}u^{\floor{(i+qj)/k}}.
\end{equation*}
 It is easy to see that $R(2k;q,i) = P(2k, q,i)/(1-u^q)$, so \eqref{e-rec2} is equivalent to
\begin{equation*}
R(2k+2q,q,i) - R(2k,q,i)=\frac{1}{1-u}.
\end{equation*}

We shall prove the equivalent formula
\begin{equation}
\label{e-rec3}
\frac{R(2k+2q,q,i)}{1-u} - \frac{R(2k,q,i)}{1-u}=\frac{1}{(1-u)^2}.
\end{equation}
The coefficient of $u^n$ in $R(2k,q,i)/(1-u)$ is the number of nonnegative integers $j$ such that $\floor{(i+qj)/k}\le n$, i.e.,
$(i+qj)/k<n+1$,
which is equivalent to
\begin{equation}
\label{e-y1}
 0\le j< \frac{k(n+1)-i}{q} .
\end{equation}
Similarly,  the coefficient of $u^n$ in $R(2k+2q,q,i)/(1-u)$ is the number of  integers $j'$  such that
\begin{equation*}
-n-1\le j'-n-1< \frac{k(n+1)-i}{q} ,
\end{equation*}
or equivalently, the number of integers $j$ such that
\begin{equation}
\label{e-y2}
-n-1\le j< \frac{k(n+1)-i}{q}.\end{equation}

To prove \eqref{e-rec3} we must show  that  if $n<0$ then \eqref{e-y1} and \eqref{e-y2} have the same number of  solutions, but if $n\ge0$ then \eqref{e-y2} has $n+1$ more solutions than \eqref{e-y1}.

We first consider the case $n<0$. If $n\le -2$ then the first inequality in \eqref{e-y1}, together with the condition $i\ge -k$, gives $j<(-k-i)/q\le0$ so there are no solutions of \eqref{e-y1} and similarly there are no solutions of \eqref{e-y2}. If $n=-1$ then \eqref{e-y1} and \eqref{e-y2} are the same.

We may now assume that $n\ge 0$. We will show that the solutions of \eqref{e-y2} are those of \eqref{e-y1} together with $-1,-2,\dots, -n$. It is sufficient to show that $-1, -2,\dots, -n $ are solutions of \eqref{e-y2}, i.e., that $\bigl(k(n+1) -i\bigr)/q>-1$.
But since $i<k+q$, we have
\begin{equation*}
\frac{k(n+1)-i}{q} \ge \frac{k-i}q >\frac{-q}q = -1.
\end{equation*}

We can reduce the case $q<0$ to the case $q>0$.
Reversing the order of summation in \eqref{e-P} gives
\[P(2k, q, i) = \sum_{j=0}^{k-1} u^{\floor {(i + q(k-1-j))/k}}.\]
Since
\begin{equation*}
\floor {\frac{(i + q(k-1-j))}{k}} = q +\floor{\frac{i-q -qj}{k}},
\end{equation*}
we have
\begin{equation}
\label{e-q5}
P(2k,q,i) = u^q P(2k, -q, i-q).
\end{equation}
Now suppose that $q<0$ and let $k'=k+q,$ $q'=-q$, and $i'=i-q$. Then the inequalities $k+q>0$ and $-k\le i < k+q$ give $k'>0$ and $-k'\le i' < k' + q'$, so by what we have already proved,
\begin{equation}
\label{e-q6}
P(2k'+2q',q',i') - P(2k',q',i')=\frac{1-u^{q'}}{1-u}.
\end{equation}
Then
\begin{align*}
P(2k+2q,q&,i) - P(2k,q,i)\\
&=u^q\bigl( P(2k+2q, -q, i-q) -P(2k, -q, i-1)\bigr),\text{ by \eqref{e-q5}}\\
  &=u^q\bigl( P(2k', q', i') -P(2k'+2q', q', i')\bigr)\\
  &= -u^q\frac{1-u^{q'}}{1-u}, \text{ by \eqref{e-q6}}\\
  &=\frac{1-u^q}{1-u}.\qedhere
\end{align*}
\end{proof}

With some additional work, which we omit here, we can show that Theorem \ref{t-Prec} is equivalent to  the following symmetric reciprocity formula generalizing Theorem  \ref{t-recip1}:
For positive integers $j$ and $k$ define
\[Q(k,j,i) = P(2k,j-k,i) = \sum_{l=0}^{k-1} u^{\floor{(i+jl)/k}-l}.\]
Then  for $0\le i< j+k$ we have
\begin{equation}
\label{e-car}
u^{j-1} Q(j,k,i) - u^{k-1} Q(k,j,i) = \frac{u^k - u^j}{1-u}.
\end{equation}
Formula \eqref{e-car} is a specialization of a result of Carlitz  \cite[equation (1.16)]{g-carlitz}. A simpler derivation of Carlitz's formula was given by Berndt and Dieter \cite[Corollary 5.8]{b-d}. Reciprocity theorems for related polynomials have been studied by Pettet and Sitaramachandrarao \cite{g-p-s}, Beck \cite{g-beck1},   Beck, Haase and Matthews \cite{g-beck2}, and Beck \cite{g-beck3}.

\end{document}